\documentclass[a4paper, 11pt, headings = small, abstract]{scrartcl}
\usepackage[english]{babel}

\usepackage{amscd,amsmath,amsthm,array,bbm,hhline,mathrsfs, enumerate,dsfont, booktabs,fancybox,calc,textcomp,xcolor,graphicx,bbm,xspace,nicefrac,stmaryrd,url,arcs,listings,nameref}
\usepackage[theoremfont]{newpxtext}
\usepackage[vvarbb, upint, bigdelims]{newpxmath}
\usepackage[scr=boondoxupr, scrscaled = 1.05, cal = euler, calscaled =1.05]{mathalpha}
\usepackage[scaled=0.95]{inconsolata}
\DeclareMathAlphabet{\mathsf}{OT1}{\sfdefault}{m}{n}
\SetMathAlphabet{\mathsf}{bold}{OT1}{\sfdefault}{b}{n}
\usepackage[utf8]{inputenc}
\usepackage{anyfontsize}
\usepackage[shortlabels]{enumitem}
\usepackage{etoolbox}
\usepackage[normalem]{ulem}
\usepackage{mathtools}
\usepackage{geometry}
\usepackage{float}

\usepackage{bm}
\usepackage{tikz}
\usepackage[outdir =./]{epstopdf}
\usepackage[citestyle = numeric-comp,bibstyle = numeric, giveninits=true, natbib = true, backend = biber,maxbibnames=99,doi=false]{biblatex}
\usepackage{authblk}
\usepackage{csquotes}
\numberwithin{equation}{section}

\usepackage{chngcntr}
\counterwithin{figure}{section}

\usepackage{xcolor}
\usepackage{hyperref}
\hypersetup{colorlinks=true, allcolors=myteal}
\definecolor{WIMgreen}{RGB}{60 134 132}
\definecolor{red_pers}{RGB}{204 37 41}
\definecolor{UMblue}{RGB}{4 47 86}
\definecolor{myteal}{RGB}{0 123 137}
\definecolor{nd}{RGB}{0 0 0}
\definecolor{dartmouthgreen}{rgb}{0.05, 0.5, 0.06}\definecolor{cobalt}{rgb}{0.0, 0.28, 0.67}\definecolor{coolblack}{rgb}{0.0, 0.18, 0.39}
\definecolor{glaucous}{rgb}{0.38, 0.51, 0.71}\definecolor{hooker\'sgreen}{rgb}{0.0, 0.44, 0.0}\definecolor{lemonchiffon}{rgb}{1.0, 0.98, 0.8}\definecolor{oucrimsonred}{rgb}{0.6, 0.0, 0.0}\definecolor{radicalred}{rgb}{1.0, 0.21, 0.37}\definecolor{raspberry}{rgb}{0.89, 0.04, 0.36}\definecolor{royalazure}{rgb}{0.0, 0.22, 0.66}
\definecolor{dex}{rgb}{0.05,0.5,0.06}

\theoremstyle{plain}
\newtheorem{theorem}{Theorem}[section]
\newtheorem{proposition}[theorem]{Proposition}
\newtheorem{lemma}[theorem]{Lemma}

\theoremstyle{definition}

\theoremstyle{remark}

\SetLabelAlign{center}{\hss#1\hss}

\def\E{\mathbb{E}}

\def\N{\mathbb{N}}

\def\N{\mathbb{N}}

\def\R{\mathbb{R}}
\definecolor{darkred}{rgb}{0,0.6,0}

\def\cN{\mathcal{N}}

\def\Var{\mathrm{Var}}
\def\Cov{\mathrm{Cov}}

\newcommand{\cF}{\mathcal{F}}

\newcommand{\ep}{\varepsilon}
\newcommand{\cO}{\mathcal{O}}

\newcommand{\sign}{\operatorname{sign}}
\newcommand{\tr}{\operatorname{tr}}

\renewcommand{\hat}{\widehat}

\restylefloat{table}
\newcommand*\samethanks[1][\value{footnote}]{\footnotemark[#1]}

\newcommand{\vertiii}[1]{{\left\vert\kern-0.25ex\left\vert\kern-0.25ex\left\vert #1
		\right\vert\kern-0.25ex\right\vert\kern-0.25ex\right\vert}}

\let\originalleft\left
\let\originalright\right
\renewcommand{\left}{\mathopen{}\mathclose\bgroup\originalleft}
\renewcommand{\right}{\aftergroup\egroup\originalright}

\newcommand{\GD}{\mkern1mu \mathsf{GD}}
\newcommand{\SGD}{\mkern1mu \mathsf{SGD}}
\newcommand{\FGD}{\mkern2mu \mathsf{FGD}}

\definecolor{myteal}{RGB}{0 123 137}
\definecolor{cs}{rgb}{0 0 0}
\definecolor{radicalred}{rgb}{1.0, 0.21, 0.37}
\definecolor{dex}{RGB}{0 0 0}
\allowdisplaybreaks
\bibliography{graddes}
\title{Improving the Convergence Rates of Forward Gradient Descent with Repeated Sampling}
\author{Niklas Dexheimer\thanks{ University of Twente \newline Faculty of Electrical Engineering, Mathematics and Computer Science \newline Drienerlolaan 5, 7522 NB Enschede, The Netherlands.\newline Email: \href{mailto:n.dexheimer@utwente.nl}{n.dexheimer@utwente.nl}/\href{mailto:a.j.schmidt-hieber@utwente.nl}{a.j.schmidt-hieber@utwente.nl}
}
\qquad Johannes Schmidt-Hieber\samethanks
}
\begin{document}
\maketitle
\begin{abstract}
Forward gradient descent (FGD) has been proposed as a biologically more plausible alternative of gradient descent as it can be computed without backward pass. Considering the linear model with $d$ parameters, previous work has found that the prediction error of FGD is, however, by a factor $d$ slower than the prediction error of stochastic gradient descent (SGD). In this paper we show that by computing $\ell$ FGD steps based on each training sample, this suboptimality factor becomes $d/(\ell \wedge d)$ and thus the suboptimality of the rate disappears if $\ell \gtrsim d.$ We also show that FGD with repeated sampling can adapt to low-dimensional structure in the input distribution. The main mathematical challenge lies in controlling the dependencies arising from the repeated sampling process. 
\end{abstract}
\section{Introduction}
Gradient descent (GD) is a cornerstone optimization method in machine learning and statistics. While traditional gradient descent relies on evaluating exact gradients, stochastic gradient descent (SGD) only evaluates the gradient for smaller batches of the training data in order to reduce the potentially enormous numerical costs of computing the exact gradient. GD and SGD are both non-local training rules as the updates of each single parameter depend on the values of all other parameters. However, it is biologically implausible that such non-local training rules underlie the learning of the brain, as this would require each neuron being informed about the state of all other neurons before each update \cite{ExcitementCrick,whit19,tava19,lillicrap20, trappenberg22}. Consequently, interest in gradient-free and in particular zeroth-order methods for optimization has emerged, aiming for a better mathematical understanding of the brain in modelling biological neural networks (BNNs). An example is the forward gradient descent (FGD) algorithm, introduced in \cite{baydin2022}.

Before presenting FGD, we firstly introduce a general supervised learning framework to illustrate its connection to zeroth-order optimization methods. For this we assume, we are given $n$ independent, identically distributed (iid) training samples $(X_1,Y_1), (X_2,Y_2),\\ \ldots , (X_n,Y_n) \in \R^d\times\R$ satisfying
\[Y_i=f_{\theta_{\star}}(X_i)+\ep_i,\quad i=1,2,\ldots,n. \]
Here $f_{\theta_{\star}} : \R^d\to\R$ is an unknown regression function  depending on a $p$-dimensional parameter vector $\theta_{\star},$ and $\ep_1,\ep_2,\ldots, \ep_n$ are iid centered noise variables. 

The standard strategy in ML is to decrease the empirical loss, defined 
as
\[\frac 1n \sum_{i=1}^n L\big(Y_i,f_{\hat{\theta}}(X_i)\big), \] 
via a form of gradient descent (GD). Given a possibly randomized initialization $\theta_0\in\R^p,$ and learning rates $\alpha_1,\alpha_2,\ldots,$ GD leads to the recursive update rule 
\begin{equation}\label{eq: gd}
\hat{\theta}^{\GD}_k=\hat{\theta}^{\GD}_{k-1}-\alpha_k \frac{1}{n}\sum_{i=1}^n\nabla_\theta L\big(Y_i,f_{\hat{\theta}^{\GD}_{k-1}}(X_i)\big),\quad k=1,\ldots, n,
\end{equation}
where $\nabla_\theta L$ denotes the gradient of $L$ with respect to the parameter vector $\theta$. Stochastic gradient descent (SGD) computes the gradient based on one training sample and is given by
\begin{equation}\label{eq: sgd}
    \hat{\theta}^{\SGD}_k=\hat{\theta}^{\SGD}_{k-1}-\alpha_k \nabla_\theta L\big(Y_k,f_{\hat{\theta}^{\SGD}_{k-1}}(X_k)\big),\quad k=1,\ldots,n.
\end{equation}
Forward gradient descent (FGD) replaces the gradient in \eqref{eq: sgd} by a noisy estimate, the so-called forward gradient, leading to the update rule
\begin{equation}\label{eq: fgd}
    \hat{\theta}^{\FGD}_k=\hat{\theta}^{\FGD}_{k-1}-\alpha_k \Big(\nabla_\theta L\big(Y_k,f_{\hat{\theta}^{\FGD}_{k-1}}(X_k)\big)\Big)^\top \xi_k\xi_k,\quad k\in[n],
\end{equation}
with $\xi_1, \ldots, \xi_n$ an iid sequence of $p$-dimensional standard normal random vectors.
For any deterministic $v\in\R^p,$
\[\E[v^\top\xi_1\xi_1]=v, \]
showing that the forward gradient is an unbiased estimate of the true gradient. Rescaling $\xi_k \to \xi_k/\lambda$ and $\alpha_k \to \alpha_k \lambda^2$ does not change the FGD updates. Therefore we do not gain anything by assuming that the covariance matrix of $\xi_k$ is a constant multiple of the identity matrix. 

The connection between the FGD estimator and zeroth-order optimization methods is not immediately obvious, as \eqref{eq: fgd} still contains the true gradient.
However, first-order Taylor approximation shows that for a continuously differentiable function $g\colon \R^p\to\R$ and $\theta\in\R^p,$
\[g(\theta+\xi)-g(\theta)\approx (\nabla g(\theta))^\top\xi. \]
Hence, up to an error term, FGD is equivalent to the zeroth-order optimization method which replaces the true gradient by the proxy
\[\big(g(\theta+\xi)-g(\theta)\big)\xi, \]
see e.g. \cite{liu20,conn09}. A connection between this zeroth-order method and Hebbian learning in BNNs (see Chapter 6 in \cite{trappenberg22} and \cite{hebb49}) has recently been shown in \cite{sh23}.

Interest in FGD is not only motivated by its connection to gradient-free optimization, but also by the fact that the forward gradient of an artificial neural network can be evaluated exactly by a single forward pass through the network via forward mode automatic differentiation (for more details see \cite{baydin2022,ren23}) and without computing the whole gradient. Backpropagation, on the contrary, relies on both forward and backward passes. 

The linear model assumes $p=d$ and $f_\theta(x)=\theta^\top x$ (see Section \ref{sec: model}). \cite{bos24} contains a first mathematical analysis of the mean squared error (MSE) of FGD in the linear model. It is shown, that for sample size $n,$ the MSE of the FGD estimator can be bounded by $d^2\log(d)/n.$ This contrasts the $d/n$ convergence rate for SGD, which is also known to be minimax optimal. This suboptimality is due to the fact that instead of the true $d$-dimensional gradient, the scalar random projection $(\nabla L)^\top\xi$ is employed in the update rule. In fact, \cite{sh23hebbian} shows that the rate $d^2/n$ is minimax optimal in the linear model with Gaussian design over the class of estimators that rely on querying each training sample at most twice. This is also in line with previous results on zeroth order optimization (see e.g. \cite{spok17,duchi15}).  

Observing different random projections for the same training sample provides more information for learning. Thus, a natural idea to improve FGD is to perform several updates per training sample. We introduce and analyze FGD($\ell$) that performs $\ell$ updates per training sample via the recursion
\begin{equation}\label{eq: fgd mult}
    \hat{\theta}^{\FGD}_{k,r}=\begin{cases}
      \hat{\theta}^{\FGD}_{k-1,\ell}-\alpha_k \bigg(\nabla_\theta\bigg(L\big(Y_k,f_{\hat{\theta}^{\FGD}_{k-1,\ell}}(X_k)\big)\bigg)\bigg)^\top \xi_{k,1}\xi_{k,1},\quad &r=1
      \\
      \hat{\theta}^{\FGD}_{k,r-1}-\alpha_k \bigg(\nabla_\theta\bigg(L\big(Y_k,f_{\hat{\theta}^{\FGD}_{k,r-1}}(X_k)\big)\bigg)\bigg)^\top \xi_{k,r}\xi_{k,r},\quad &r=2,\ldots, \ell,
    \end{cases}
    \ \ \text{for} \  k=1,2,\ldots,
\end{equation}
where the $\xi_{k,r}$ are again independent $p$-dimensional standard normal random variables. In particular, FGD($1$) coincides with FGD defined in \eqref{eq: fgd}.

FGD($\ell$) has already been conjectured to improve the performance of FGD in the original article \cite{baydin2022}, however its mathematical analysis is significantly more involved due to the stochastic dependence induced by the repeated sampling. Indeed for $r\geq2,$ the components of the forward gradient used in the update rule $\hat{\theta}^{\FGD}_{k,r-1},X_k,Y_k$ are not independent as in the original FGD algorithm \eqref{eq: fgd}. 

Implementation of FGD($\ell$) only requires biologically plausible forward passes and no backwards passes. From a biological point of view this corresponds to the idea that the brain uses the same data repeatedly for learning. In a stochastic process context repeated sampling has been investigated in \cite{christensen24}. A downside of repeated sampling is the increase in runtime in the number of repetitions per sample $\ell$.

In this paper we investigate the mean squared prediction error of FGD($\ell$) \eqref{eq: fgd mult} in the linear model. For $$\Sigma:=\E[X_1X_1^\top]$$ the second moment matrix of the iid covariate vectors, our main contributions are:
\begin{itemize}
    \item We show that FGD($\ell$) achieves the convergence rate \[\frac{d^2}{(\ell\land d)k}.\] For $\ell=d,$ the minimax rate $d/k$ is achieved. The convergence rate is derived by carefully tracking the dependencies in the update rule \eqref{eq: fgd mult}.
    \item We bound the mean squared prediction error (MSPE) instead of the MSE. This allows us to handle the case where the second moment matrix $\Sigma$ is degenerate. In particular, we show that if the covariate vectors are drawn from a distribution that is supported on an $s$-dimensional linear subspace, the convergence rate of FGD($s$) can be bounded by $s/n,$ with $s$ the rank of $\Sigma$.
    \item Rate-optimal results for the MSE are obtained if $\Sigma$ is non-degenerate.
\end{itemize}
The paper is structured as follows. The setting is formally introduced in Section \ref{sec: model}. Section \ref{sec: mspe} states and discusses the convergence rate for the MSPE of FGD($\ell$). The paper closes with a brief simulation study in Section \ref{sec: sim} and a conclusion (Section \ref{sec: conc}). 

\subsection{Notation}\label{subsec: not}
Let $d,d_1,d_2\in\N$. For a matrix $A\in\R^{d_1\times d_2},$ write $A^\top$ for its transpose. For a symmetric, positive semi-definite matrix $A\in\R^{d\times d},$ we denote its maximal (minimal) eigenvalue by $\lambda_{\max}(A)$ ($\lambda_{\min}(A)$), its smallest non-zero eigenvalue by $\lambda_{\min\neq 0}(A)$ and its condition number by $\kappa(A)=\lambda_{\max}(A)/\lambda_{\min}(A)$. The Loewner order for symmetric matrices $A_1,A_2\in\R^{d\times d}$ is denoted by $A_1\geq_{\mathsf{L}}A_2$ and means that the matrix $A_1-A_2$ is positive semi-definite.
Furthermore we define $[d]\coloneq\{1,\ldots,d\},$ denote by $\mathbb{I}_d$ the $d\times d$ identity matrix and write $\Vert A\Vert=\sqrt{\lambda_{\max}(A^\top A})$ for the spectral norm of a matrix $A\in\R^{d_1\times d_2}$. The expectation operator $\E$ refers to the expectation taken with respect to all randomness. Lastly, we denote the $d$-dimensional normal distribution with mean vector $\mu\in\R^d$ and covariance matrix $\Sigma\in\R^{d\times d}$ by $\mathcal{N}(\mu,\Sigma)$. By convention, the empty sum is $0$ and the empty product is $1$.
\section{Mathematical Model and Bias}\label{sec: model}
Throughout the paper we assume iid training data $(X_1,Y_1), (X_2,Y_2),\ldots \in\R^d\times \R$ following the linear model with true parameter $\theta_\star\in\R^d,$ that is,
\[ Y_i=X_i^\top\theta_\star+\ep_i,\quad i\in\N, \]
with $\ep_1,\ep_2,\ldots$ an iid sequence of random variables, independent of the covariate vectors $X_1, X_2,\ldots$ that satisfy $\E[\ep_1]=0$ and $\Var(\ep_1)=1$. We do not assume that the covariate vectors $X_i$ are centered, but will impose boundedness conditions on $\Vert X_i\Vert.$ Our interest lies in estimating the true parameter $\theta_\star,$ and bounding the mean squared prediction error (MSPE) of the estimator $\hat{\theta},$ which is given by
\[\operatorname{MSPE}(\hat{\theta})\coloneq \E[(X^\top(\theta_{\star}-\hat{\theta}))^2], \]
where $X$ is an independent copy of a training input. As the gradient of the function 
\[L\colon \R^d\to[0,\infty),\quad \theta\mapsto \frac 12 \big(u^\top(\theta_\star-\theta)\big)^2, \]
for $u\in\R^d,$ satisfies
\[\nabla L(\theta)=-(u^\top(\theta_\star-\theta))u, \]
the FGD($\ell$) updates are given by
\begin{align*}
\theta_{k,r}=\begin{cases}
\theta_{k-1,\ell}+\alpha_k(Y_k-X_k^\top\theta_{k-1,\ell})X_k^\top \xi_{k,r}\xi_{k,r},\quad &r=1,\\
\theta_{k,r-1}+\alpha_k(Y_k-X_k^\top\theta_{k,r-1})X_k^\top \xi_{k,r}\xi_{k,r},\quad &r\in[\ell]\backslash\{1\},\\
\end{cases}
\end{align*}
where $(\xi_{k,r})_{k\in\N,r\in[\ell]}$ is a collection of iid random vectors, such that $\xi_{1,1}\sim\cN(0,\mathbb{I}_d),$ and $\theta_{0,\ell}=\theta_0\in\R^d,$ is a (possibly random) initialization independent of everything else. A standard choice is to take $\theta_0$ to be equal to the zero vector in $\R^d$. Setting $\theta_{k,0}\coloneq\theta_{k-1,\ell},$ for $k\in\N$ allows us to write the FGD update rule as 
\begin{align*}
\theta_{k,r}=
\theta_{k,r-1}+\alpha_k(Y_k-X_k^\top\theta_{k,r-1})X_k^\top \xi_{k,r}\xi_{k,r},\quad r\in[\ell], \ k=1,2,\ldots
\end{align*}
If the noise vector $\xi_{k,r}$ happens to lie close to the linear space spanned by the gradient direction $X_k,$ that is, $\xi_{k,r} \approx \lambda X_k$ for some $\lambda \neq 0,$ then, $X_k^\top \xi_{k,r}\xi_{k,r} \approx \lambda^2 X_k^\top X_k X_k$ and the update is close to a SGD update with learning rate $\approx \alpha_k \lambda^2 X_k^\top X_k.$ If $X_k$ and $\xi_{k,r}$ are nearly orthogonal, then $X_k^\top \xi_{k,r} \approx 0$ and $\theta_{k,r}\approx 
\theta_{k,r-1}.$ Thus, little information about the gradient direction $X_k$ in the noise vector $\xi_{k,r}$ leads to a small FGD parameter update.

It seems natural to absorb the scalar product $X_k^\top \xi_{k,r}$ in the learning rate. Since the learning rate has to be positive, rescaling the learning rate by $\beta_k=\alpha_k  |X_k^\top \xi_{k,r}|/\|X_k\|$ yields the update formula 
\begin{align}
\theta_{k,r}=
\theta_{k,r-1}+\beta_k(Y_k-X_k^\top\theta_{k,r-1})\|X_k\| \sign(X_k^\top \xi_{k,r})\xi_{k,r},\quad r\in[\ell], \ k=1,2,\ldots
    \label{eq.aFGD}
\end{align}
with $\sign(u)=\mathbf{1}(u>0)-\mathbf{1}(u<0)$ the sign function. Since $\|X_k\|\E[\sign(X_k^\top \xi_{k,r})\xi_{k,r}\vert X_k]=\sqrt{2/\pi}X_k,$ also this scheme does, in expectation, gradient descent. Because of $X_k^\top \xi_{k,r}/\|X_k\| \, |X_k\sim \mathcal{N}(0,1),$ the learning rates $\beta_k$ and $\alpha_k$ are of the same order.

We refer to the update rule \eqref{eq.aFGD} as adjusted forward gradient descent or aFGD($\ell$). By construction, the aFGD updates are independent of $|X_k^\top \xi_{k,r}|.$ If $\xi_{k,r}\approx \lambda X_k,$ aFGD is approximately an SGD update with learning rate $\beta_k \|X_k\| \lambda.$ The theoretical guarantees (Appendix \ref{app: afgd}) and the simulations (Section \ref{sec: sim}) show that FGD($\ell$) and aFGD($\ell$) behave similarly.

\subsection{Bias}
To get a first impression of the behavior of FGD($\ell$) compared to FGD and SGD, we compute the bias of the estimator. The linear model assumption $Y_k=X_k^\top\theta_k+\ep_k,$  implies for $k\in\N,r\in[\ell],$
\begin{align*}
\theta_{k,r}-\theta_{\star}=\big(\mathbb{I}_d- \alpha_k \xi_{k,r}\xi_{k,r}^\top X_k X_k^\top\big) \big(\theta_{k,r-1}-\theta_{\star}\big)+ 
\alpha_k \ep_k X_k^\top \xi_{k,r}\xi_{k,r},\quad r\in[\ell].
\end{align*}
Hence, defining the product $\prod_{r=1}^\ell A_r$ as $A_\ell \cdot \ldots \cdot A_1,$ for matrices $A_1,\ldots,A_\ell$ of suitable dimensions, we can write 
\begin{align}
\begin{split}
\theta_{k,r}-\theta_{\star}= &\Big(\prod_{u=1}^r \big(\mathbb{I}_d- \alpha_k \xi_{k,u}\xi_{k,u}^\top X_k X_k^\top\big)\Big) \big(\theta_{k-1,\ell}-\theta_{\star}\big) \\
&+ 
\alpha_k \ep_k \sum_{u=1}^r \Big(\prod_{s=u+1}^r \big(\mathbb{I}_d- \alpha_k \xi_{k,s}\xi_{k,s}^\top X_k X_k^\top\big)\Big) X_k^\top \xi_{k,u}\xi_{k,u}.
\end{split}
\label{eq.iwck743}
\end{align}
The identity above decomposes the FGD($\ell$) updates into a gradient descent step perturbed by a stochastic error and resembles a vector autoregressive process of lag one. Such an autoregressive representation has also been shown to be central to analyze SGD with dropout regularization in the linear model \cite{clara24,li24}.
As $\ep_k$ and $\theta_{k-1,\ell}$ are independent of $X_k$ and $(\xi_{k,r})_{r\in[\ell]},$ the representation \eqref{eq.iwck743} directly implies the following result for the bias of FGD($\ell$).
\begin{theorem}\label{thm: bias}
For any positive integers $k,\ell,$
\begin{align*}
	\E[\theta_{k,\ell}]-\theta_\star
	&=\left(\prod_{i=1}^k\E\big[(\mathbb{I}_d-\alpha_iX_iX_i^\top)^{\ell}\big]\right)\big(\E\left[\theta_{0}\right]-\theta_\star\big).
\end{align*}
\end{theorem}
The case $\ell=1$ follows from Theorem 3.1 of \cite{bos24}. In this case the bias agrees with the bias of SGD in the linear model. Hence, due to the bias--variance decomposition, the suboptimal convergence rate of FGD is caused by the increased variance that is due to using a noisy estimate instead of the true gradient. Nevertheless for $\ell>1$, expanding the $\ell$-th power yields
\[(\mathbb{I}_d-\alpha_iX_iX_i^\top)^\ell = \mathbb{I}_d-\alpha_i\ell X_iX_i^\top +\cO(\alpha_i^2),\]
suggesting that the bias for repeated sampling FGD in Theorem \ref{thm: bias} is approximately
\[\left(\prod_{i=1}^k\E[\mathbb{I}_d-\alpha_i\ell X_iX_i^\top]\right)\big(\E\left[\theta_{0}\right]-\theta_\star\big).\]
Consequently, FGD($\ell$) with learning rate $\ell^{-1}\alpha_i$ achieves a similar bias as FGD with learning rate $\alpha_i$. As mentioned in the introduction, reducing the learning rate by a factor $\ell^{-1}$ corresponds to reducing the variance of the noise variables $\xi_{k,r}$ by a factor $\ell^{-1}$. Thus Theorem \ref{thm: bias} suggests that FGD($\ell$) has the same bias as FGD but reduces the variance. This is proved in the next section.
\section{Prediction error}\label{sec: mspe}
The following result derives nonasymptotic convergence guarantees for FGD($\ell$). The proof is deferred to Appendix \ref{subsec: proof}. Recall that $\Sigma:=\E[X_1X_1^\top]$ and $\lambda_{\min\neq 0}(\Sigma)$ denotes the smallest non-zero eigenvalue of $\Sigma.$
\begin{theorem}\label{thm: mspe}
	Let $k\in\N$ and assume $\Vert X_1\Vert^2\leq b$ for some $b>0$. If the learning rate is of the form
	\[\alpha_i=\frac{c_1}{\ell(c_1c_2+i)},\quad i=1,\ldots,k, \]
	for constants $c_1,c_2>0$ satisfying
	\[c_1\geq \frac{2}{\lambda_{\min\neq 0}(\Sigma)},\quad c_2\geq \frac{3b}{\lambda_{\min\neq 0}(\Sigma)}\bigg(2\lambda_{\max}(\Sigma)+\frac{\tr(\Sigma)}{\ell}\bigg),  \]
	then,
	\begin{align*}
		\operatorname{MSPE}(\theta_{k,\ell})
		\leq\left(\frac{1+c_1c_2}{k+1+c_1c_2}\right)^2\operatorname{MSPE}(\theta_{0}) 
		+16bc_1^2\bigg(\lambda_{\max}(\Sigma)
		+\frac{\tr(\Sigma)}{\ell}\bigg)\frac{k}{(k+1+c_1c_2)^2}.
	\end{align*}
\end{theorem}
A similar result holds for aFGD($\ell$), see Theorem \ref{thm: afgd mspe}. A consequence is that 
\[c_1=\frac{2}{\lambda_{\min\neq0}(\Sigma)},\quad c_2=\frac{4b}{\lambda_{\min\neq 0}(\Sigma)}\bigg(2\lambda_{\max}(\Sigma)+\frac{\tr(\Sigma)}{\ell}\bigg), \]
and $\ell\geq \tr(\Sigma)$ yield the  convergence rate
\begin{align}\label{eq: mspe remark}
	&\operatorname{MSPE}(\theta_{k,\ell})
\lesssim\frac{b^2}{k^2}\operatorname{MSPE}(\theta_{0})
	+\frac{b}{k}.
\end{align}
If $\Vert X_1\Vert_\infty\leq m$ almost surely, then one can choose $b=dm^2$ and as  $\operatorname{MSPE}(\theta_{0})$ is also of order $d,$ we thus recover the minimax rate $d/k$ whenever $k\gtrsim d^2$. In particular, if $\Sigma$ has full-rank, the same convergence rate for the mean squared error ($\operatorname{MSE}$) holds, since then $\lambda_{\min}(\Sigma)>0$ and thus
\[\frac{\operatorname{MSPE}(\theta_{k,\ell})}{\lambda_{\min}(\Sigma)}=\frac{\E[(\theta_{k,\ell}-\theta_{\star})^\top \Sigma(\theta_{k,\ell}-\theta_{\star})]}{\lambda_{\min}(\Sigma)}\geq \E[\Vert\theta_{k,\ell}-\theta_{\star}\Vert^2 ]= \operatorname{MSE}(\theta_{k,\ell}).\]

The prediction error also provides insights if $\Sigma$ is singular and thus $\lambda_{\min}(\Sigma)=0$. As example suppose the covariate vectors are contained in an $s$-dimensional hyperplane in the sense that
\[X_k=UZ_k,\quad k\in\N, \]
for some $U\in\R^{d\times s}, s\in[d]$ and $Z_k\in\R^s$ are iid random variables, such that $\Vert Z_1\Vert_\infty\leq m_Z$ for $m_Z>0$. In this case
\(\Vert X_1\Vert^2\leq \Vert U\Vert^2 \Vert Z_1\Vert^2\leq \Vert U\Vert^2 m_Z^2s, \) and thus \eqref{eq: mspe remark} yields
\begin{align*}
	&\operatorname{MSPE}(\theta_{k,\ell})
\lesssim\frac{ds^2}{k^2}
+\frac{s}{k}.
\end{align*}
For $k\gtrsim ds$ we obtain the rate $s/k,$ which can be much faster than $d/k$.

Lastly, we want to highlight the explicit dependence of the upper bound in Theorem \ref{thm: mspe} on the number $\ell$ of updates per training sample. This provides some insights on how to choose $\ell$ in practice. While increasing $\ell$ reduces the MSPE it also blows up the computational cost. Now Theorem \ref{thm: mspe} indicates that increasing $\ell$ beyond the so-called intrinsic dimension or effective rank $\tr(\Sigma)/\lambda_{\max}(\Sigma)$ \cite{kolt17,hdp} does not reduce the MSPE significantly. As the intrinsic dimension is obviously always smaller than the true dimension $d$, it is in particular never beneficial to repeat each sample more than $d$ times in forward gradient descent. As in the example above, the intrinsic dimension can be much smaller than the true dimension $d$ for distributions concentrated close to lower-dimensional subspaces (see Remark 5.6.3 in \cite{hdp}), implying that in these cases less updates are required for achieving rate-optimal statistical accuracy.

\section{Simulation Study}\label{sec: sim}
We study and compare the generalization properties of FGD($\ell$) on data generated from the linear model. The code is available on Github \cite{code}.

If not mentioned otherwise the inputs/covariate vectors $X_i$ will be generated according to a uniform distribution on $[-\sqrt{3},\sqrt{3}]^d$ such that $\Sigma=\mathbb{I}_d$. The true regression parameter $\theta_\star$ is drawn from a uniform distribution on $[-10,10]^d$ if not stated otherwise. Every experiment is repeated $10$ times and in the respective plots the average error is depicted. The shaded areas correspond to error $\pm$ standard deviation in regular scale plots and log(error) $\pm$ standard deviation/(error $\times \log(10)$) in log-log plots (see e.g. \cite{baird1995experimentation}).

\begin{figure}
    \centering
    \includegraphics[width=0.45\linewidth]{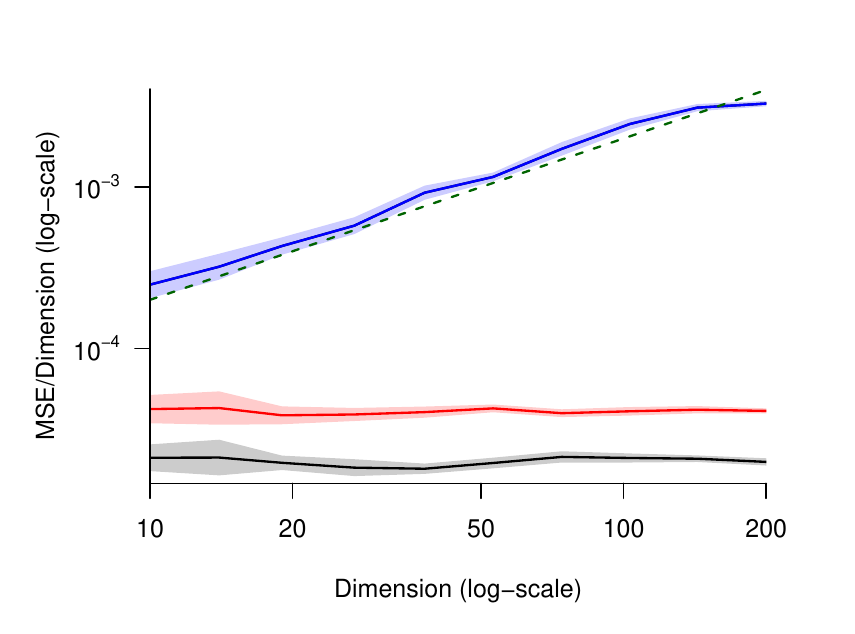}
     \includegraphics[width=0.53\linewidth]{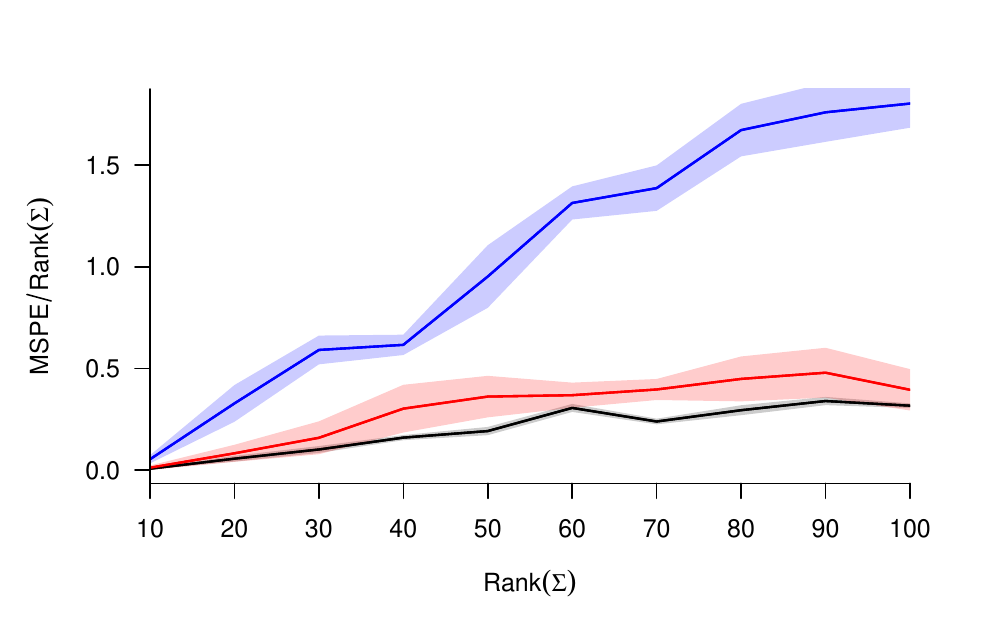}
    \caption{(Left) Log-log plot showing the dimension dependence of MSE/dimension based on $10$ repetitions of SGD (black), FGD (blue) and FGD($d$) (red). As reference, the function $d\mapsto d/n$ has been added to the plot (green, dashed). \\(Right) The functions rank($\Sigma$) $\mapsto$ MSPE/rank($\Sigma$) ($\pm$ one standard deviation) for SGD (black), FGD (blue) and FGD(rank($\Sigma$)) (red).}
    \label{fig: dimcomp}
\end{figure}

The derived theory highlights the quadratic dimension dependence of the MSE for FGD versus the linear dimension dependence of both FGD($d$) and SGD. To visualize the dimension dependence of the different methods, we choose sample size $n=5\times 10^4$ and dimensions $d$ on a logarithmic grid between $10$ and $200$. To avoid any dimension dependence of $\Vert \theta_\star\Vert,$ we first generate a draw from the uniform distribution on $[-10,10]^d$ and then take $\theta_\star$ as the $L^2$-normalization of this draw which ensures that $\Vert\theta_\star\Vert=1$. The results are summarized in Figure \ref{fig: dimcomp}.
\begin{figure}
    \centering
    \includegraphics[width=0.75\linewidth]
    {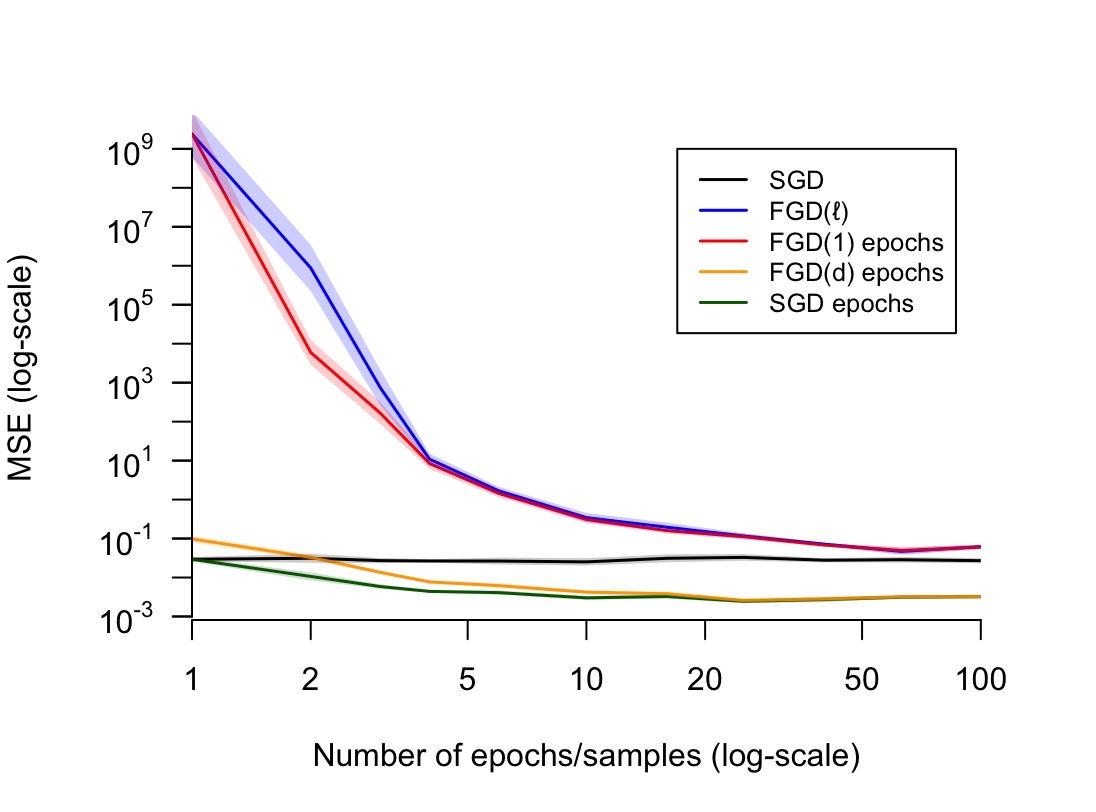}
    \\
    \includegraphics[width=0.75\linewidth]{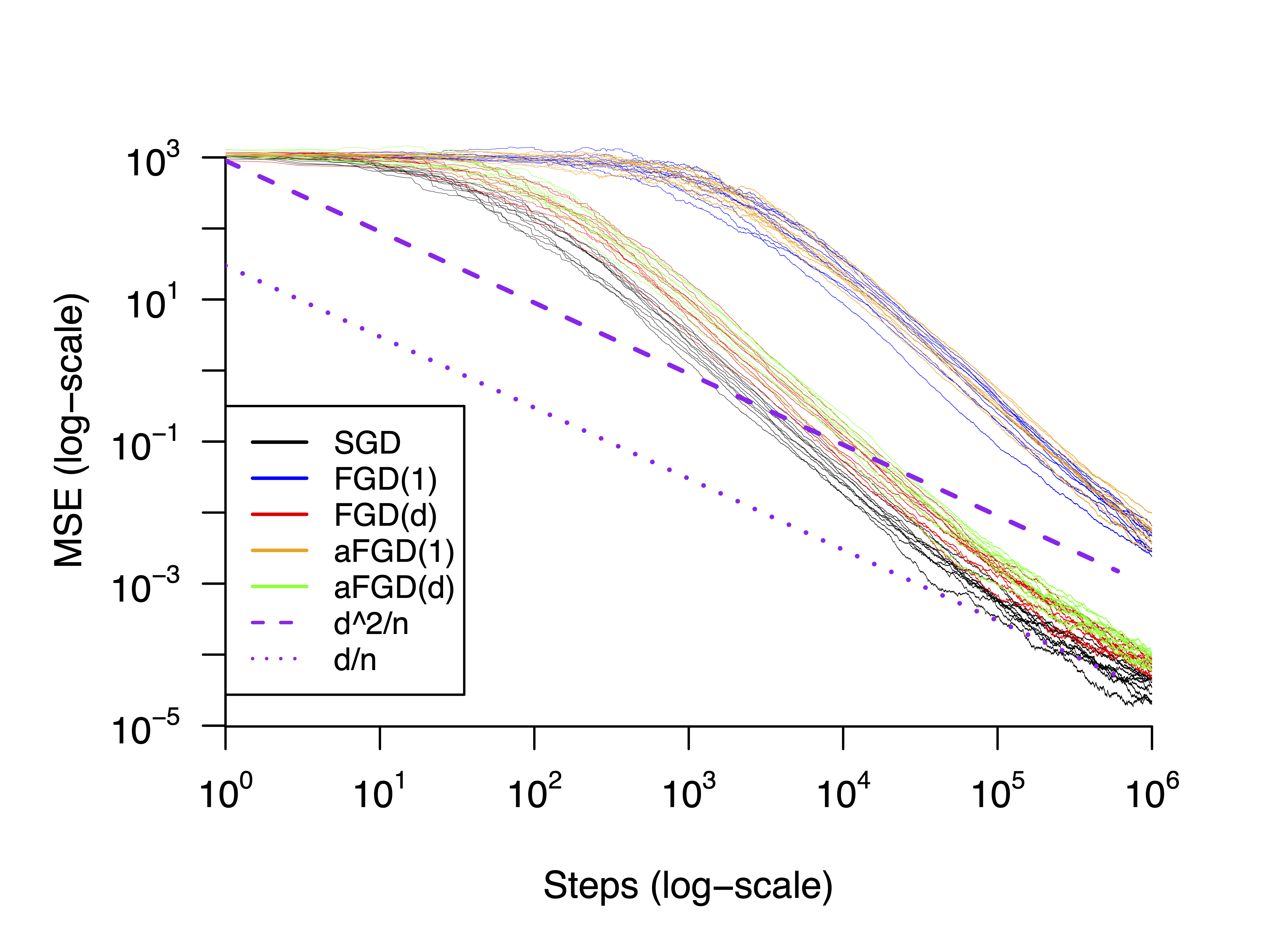}
    \caption{(Top) Log-log plot showing the dependence on $\ell$ of the averaged MSE of FGD($\ell$) (blue) as well as of FGD (red), FGD($d$) (orange) and SGD (green)  with $\ell$ epochs. The average is based on $10$ repetitions. As a benchmark also the MSE of SGD (black) with one epoch is displayed. \\(Bottom) Log-log plot showing the decay of the MSEs of SGD (black), FGD (blue), FGD($d$) (red), aFGD($1$) (orange) and aFGD($d$) (green) per step. A 'step' refers here to all the updates based on one training sample. Each line in the plot corresponds to one realization.  The lines $n\mapsto d^2/n$ (purple, dashed) and $n \mapsto d/n$ (purple/dotted) are added as a reference to allow for comparison with the derived convergence rates.}
    \label{fig: querycomp}
\end{figure}
In accordance with the derived theory, we observe that the ratio MSE/dimension stays constant for FGD($d$) and SGD, whereas the error increases linearly for FGD, matching the rate $d/n$.

SGD can often be improved by cycling several times through the dataset. Each cycle is referred to as an epoch. While similar to the repeated sampling in FGD($\ell$), this introduces a different dependence structure. Indeed, after one epoch, the SGD iterates depend on the whole sample, whereas the FGD($\ell$) iterates only depend on the already seen training samples. Due to differences in the dependence structure, the statistical analysis of SGD or FGD with multiple epochs requires distinct approaches, building on mathematical techniques developed for gradient descent with epochs \cite{chaud19,hazan14,sekhari2021sgd,yan20}, with diffusion approximation emerging as a particularly promising direction \cite{ankirchner22,mandt2015continuous}. However, since using SGD with multiple epochs is standard in applications, we compare both FGD($\ell$) and FGD($1$) with $\ell$ epochs, and FGD($d$) and SGD based on multiple epochs to standard SGD.
Additionally, we examine the trade-off between computational cost and generalization performance by studying the function $\ell \mapsto$ MSE(FGD($\ell$)). The results are summarized in Figure \ref{fig: querycomp}. Here we use sample size $n=10^5$, dimension $d=30$ and $\ell$ as well as the number of epochs chosen equidistantly on a logarithmic grid from $1$ to $100$. The MSE of FGD($\ell$) drops rapidly for small $\ell$ but slows down for $\ell\geq 30$ in accordance with our theoretical results. Similarly, the MSE of SGD and FGD($d$) decreases if up to $10$ epochs are used and remains essentially constant afterwards. SGD and FGD($d$) have moreover nearly the same MSE for larger number of epochs. The same is true for the FGD($\ell$) and FGD($1$) with $\ell$ epochs.

Another interesting aspect of Theorem \ref{thm: mspe} is the dependence of the prediction error on the rank of the second moment matrix $\Sigma$. For $d=100$ and sample size $n=3\times 10^4$, we compare the MSPE of SGD, FGD and FGD($s$) if the inputs/covariate vectors $(X_i)_{i\in\N}$ are of the form
\[X_i=UZ_i,\quad i=1,2,\ldots , \]
with $U\in\R^{d\times s},$ $s\in\{10,20,\ldots,100\},$ and $Z_1,Z_2,\ldots$ independently drawn from the uniform distribution on $[-\sqrt{3},\sqrt{3}]^s$. In accordance with Theorem \ref{thm: mspe}, we use $s$ updates per training sample in FGD. Figure \ref{fig: dimcomp} shows the function $s\mapsto $ MSPE/$s$. For FGD, we observe a near linear increase, whereas SGD and FGD($s$) behave very similarly and remain almost constant for $s>50.$ The linear increase of FGD and the stability of SGD and FGD($s$) agree with Theorem \ref{thm: mspe}. This also shows that if $\Sigma$ has rank $s$, FGD only requires $s$ instead of $d$ updates per sample to match the performance of SGD. We do not have an explanation for the slow linear growth of $s\mapsto$ MSPE/$s$ that can be observed for SGD and FGD($s$) if $s$ is small.

Lastly, we compare the MSEs of SGD, FGD, FGD($d$), aFGD($1$), and aFGD($d$) as the number of steps increases. Here a step refers to all the updates based on one training sample. While SGD, FGD and aFGD($1$) compute one update per training sample, FGD($d$) and aFGD($d$) compute $d$ updates per training sample. Choosing $d=30,$ we plot the MSEs for each step. The right panel of Figure \ref{fig: querycomp} displays the $5\times 10$ functions. Again, one can observe that SGD and FGD($d$) behave very similarly. After a nearly flat initialization phase of length approximately $d,$ both SGD and FGD($d$) start decreasing and after $\approx 10^5$ steps, the MSEs are near the $n\mapsto d/n$ line that indicates the minimax rate. FGD requires a longer initialization phase, and nearly matches the line $n\mapsto d^2/n$ for large $n$ as predicted in \cite{bos24}. For aFGD, aFGD($d$) we multiply the learning rates of FGD, FGD($d$) by $\sqrt{\pi/2}$ (see Section \ref{sec: model}). Then aFGD($1$) and FGD as well as aFGD($d$) and FGD($d$) behave almost identically.

\section{Conclusion}\label{sec: conc}

In this work, we have extended forward gradient descent \cite{baydin2022} to repeated sampling forward gradient descent FGD($\ell$). We showed that repeated sampling reduces the excessively high variance of FGD. In particular, we proved that FGD($d$) achieves the same minimax optimal convergence rate as stochastic gradient descent (SGD). The key mathematical challenge is to deal with the stochastic dependencies in the update rule that emerge through the repeated sampling. The main result Theorem \ref{thm: mspe} bounds the MSPE. It reveals the influence of the learning rate and the number of repetitions per sample and can be used to choose these quantities in practice. The convergence rate of the MSPE does not improve for more than $d$ repetitions per sample. Low-dimensional structure of the input distribution is ubiquitous in modern ML applications. We account for this by proving that if the inputs/covariates are supported on a lower dimensional linear subspace, faster convergence rates can be achieved and less repeated sampling in FGD($\ell$) is necessary to obtain the optimal convergence rates.

Going beyond the linear model is challenging. A natural extension are single-index models (see e.g. \cite{bietti22,hsu18}) with either polynomial \cite{wu24} or ReLU \cite{solta17} link function.


Mimicking a school curriculum, in curriculum learning \cite{CL, soviany2022curriculum}, the order of the training data in the updates is not arbitrary, as for example in FGD($\ell$), but sorted according to a suitable definition of difficulty. Theory for gradient based curriculum learning in the linear model has been developed in \cite{wein20,weinshall18a,xu2022statistical}. We do believe that one can improve the performance of FGD($\ell$) by selecting the next training sample from the dataset based on the knowledge of the current iterate. The specific sampling strategy will also depend whether or not the random vectors $\xi_{k,r},$ $r=1,\ldots,\ell$ are released before choosing $(X_k,Y_k).$ In case they are known, one will pick a $X_k$ that is highly correlated with the noise and we assume that in this case even FGD without repeated sampling will achieve the optimal rate. If the selection of the training sample $(X_k,Y_k)$ is only allowed to depend on the previous iterate $\theta_{k-1,\ell}$, the FGD update rule $\theta_{k,r}=
\theta_{k,r-1}+\alpha_k(Y_k-X_k^\top\theta_{k,r-1})X_k^\top \xi_{k,r}\xi_{k,r}$ suggests to choose $(X_k,Y_k)$ to maximize $Y_k-X_k^\top\theta_{k-1,\ell}.$ Deriving theoretical guarantees for such schemes will be considerably more involved.

In conclusion, our findings suggest that performing many forward passes can replace the biologically implausible and memory-intensive backward pass in backpropagation without losing in terms of generalization guarantees.

\section{Acknowledgments}

This work is supported by ERC grant A2B (grant agreement no. 101124751). Part of this work has been carried out while the authors visited the Simons Institute for the Theory of Computing in Berkeley.
\appendix

\section{Proving Theorem \ref{thm: mspe}}\label{subsec: proof}
The main idea of the proof is to obtain recursive identities for the MSPE and other quantities conditioned on the currently used datum $X_k$, and subsequently using that $\theta_{k-1,\ell}$ and $X_k$ are independent. The following technical lemma extends Lemma 4.1 in \cite{bos24} to conditional expectations.
\begin{lemma}\label{lemma: cov}
Let $U$ be a $d$-dimensional random vector, $\Gamma\in\R^{d\times d}$ be  symmetric and positive definite, $Z\sim\mathcal{N}(0,\Gamma)$ and $\cF$ be a $\sigma$-algebra over $\R^d$. If $U$ is $\cF$-measurable and $Z$ is independent of $\cF$, then 
\[\E[(U^\top Z)^2 ZZ^\top\vert \cF ]=2\Gamma UU^\top\Gamma+U^\top\Gamma U\Gamma. \]
\end{lemma}
\begin{proof}
For any $i,j\in[d]$,
\begin{align*}
((U^\top Z)^2 ZZ^\top)_{i,j}&=(U^\top Z)^2 Z_iZ_j
\\
&= Z_iZ_j \sum_{l,m=1}^d U_lU_mZ_mZ_l.
\end{align*}
Hence, it follows by assumption 
\begin{align*}
(\E[(U^\top Z)^2 ZZ^\top \vert \cF ])_{i,j}&= \sum_{l,m=1}^d U_lU_m\E[Z_mZ_lZ_iZ_j]
\\
&= \sum_{l,m=1}^d U_lU_m(\Gamma_{l,m}\Gamma_{i,j}+\Gamma_{l,i}\Gamma_{m,j}+\Gamma_{l,j}\Gamma_{m,i}),
\end{align*}
where we applied Isserlis' Theorem in the last step.
Arguing as in the proof of Lemma 4.1 in \cite{bos24}, this concludes the  proof.
\end{proof}

\begin{lemma} \label{lemma: ep}
	Let $k\in\N$ and $r=0,1,\ldots,\ell.$ If $\alpha_k X_k^\top X_k\leq \tfrac 34,$ then,
    \begin{align*}
    \E[\ep_k\theta_{k,r} \vert X_k] X_k^\top=\frac{1}{X_k^\top X_k}X_k^\top \E[\ep_k\theta_{k,r} \vert X_k]
    X_k X_k^\top 
    \leq_{\mathsf{L}} 2\alpha_k r X_k X_k^\top
    \end{align*}
    and
	\begin{align*}
		&0\leq X_k^\top \E[\ep_k\theta_{k,r} \vert X_k]\notag
		=1-\big(1-\alpha_k X_k^\top X_k\big)^r \leq \min\big(1, 2\alpha_kr X_k^\top X_k\big).
	\end{align*}
\end{lemma}
\begin{proof}
If $r=0,$ we use that $\theta_{k,0}=\theta_{k-1,\ell}$ and thus $\ep_k$ and $\theta_{k,0}$ are independent. Thus, $X_k^\top\E[\ep_k\theta_{k,r} \vert X_k]=0,$ proving the case $r=0.$

For $r=1,\ldots, \ell,$ we deploy representation \eqref{eq.iwck743},
\begin{align*}
\begin{split}
\theta_{k,r}-\theta_{\star}= &\Big(\prod_{u=1}^r \big(\mathbb{I}_d- \alpha_k \xi_{k,u}\xi_{k,u}^\top X_k X_k^\top\big)\Big) \big(\theta_{k-1,\ell}-\theta_{\star}\big) \\
&+ 
\alpha_k \ep_k \sum_{u=1}^r \Big(\prod_{s=u+1}^r \big(\mathbb{I}_d- \alpha_k \xi_{k,s}\xi_{k,s}^\top X_k X_k^\top\big)\Big) X_k^\top \xi_{k,u}\xi_{k,u}.
\end{split}
\end{align*}
Using that $\ep_k$ and $\theta_{k-1,\ell}$ are mutually independent and independent of $X_k,$ $\E[\ep_k]=0,$ $\E[\ep_k^2]=1,$ $\E[\xi_{k,s}\xi_{k,s}^\top]=\mathbb{I}_d$ and $\E[X_k^\top \xi_{k,u}\xi_{k,u} \vert X_k]=X_k,$ we obtain
	\begin{align*}
		\E[\ep_k\theta_{k,r} \vert X_k]&=\E[\ep_k(\theta_{k,r}-\theta_{\star}) \vert X_k]=\alpha_k\sum_{u=1}^r (1-\alpha_kX_k^\top X_k)^{r-u} X_k =\alpha_k\sum_{i=0}^{r-1}(1-\alpha_kX_k^\top X_k)^iX_k.
	\end{align*}
From this identity, we can deduce the first assertion of the lemma. To prove the second statement, observe that
\begin{align*}
    0\leq X_k^\top \E[\ep_k\theta_{k,r} \vert X_k]
    = \alpha_k X_k^\top X_k\sum_{i=0}^{r-1}(1-\alpha_kX_k^\top X_k)^i 
    = 1-(1-\alpha_kX_k^\top X_k)^r\leq 1,
\end{align*}
using the assumption $\alpha_kX_k^\top X_k\leq 3/4$ for the last inequality.

It remains to prove 
\begin{align}
1-(1-\alpha_kX_k^\top X_k)^{r-1}\leq 2\alpha_k(\ell-1)X_k^\top X_k. \label{eq: 1}
\end{align}

Now the well-known inequality $1+x\leq \exp(x)$ that holds for all $x\in\R,$ together with
\[\exp(-x)\leq1-\frac{x}{2}, \quad x\in[0,\tfrac{3}{2}], \]
(see 4.2.37 in \cite{abra72}) and the assumption $r\leq \ell$ imply for $y\in [0,\tfrac 34],$ $1-(1-y)^{r-1}\leq 1-\exp(-2(r-1)y)\leq 2(r-1)y\leq 2(\ell-1)y.$ Thus, since $\alpha_kX_k^\top X_k\leq 3/4,$ setting $y:=\alpha_kX_k^\top X_k$ yields \eqref{eq: 1}, completing the proof.
\end{proof}

As the prediction error satisfies
\[\operatorname{MSPE}(\theta_{k,r})=\E[X^\top \E[(\theta_{k,r}-\theta_\star)(\theta_{k,r}-\theta_\star)^\top] X], \]
with $X$ an independent draw from the input distribution, it suffices to bound the expectation of the error matrix $(\theta_{k,r}-\theta_\star)(\theta_{k,r}-\theta_\star)^\top$ in the sense of Loewner, which we denote by $\leq_{\mathsf{L}}$ (recall Section \ref{subsec: not}). A first step is the following result. 
\begin{proposition}\label{prop: rec 1}
	For given $r=1,\ldots,\ell,$
\begin{align*}
	&\E[(\theta_{k,r}-\theta_{\star})(\theta_{k,r}-\theta_{\star})^\top \vert X_k]
	\\
	&\leq_{\mathsf L}
	(\mathbb{I}_d-\alpha_kX_kX_k^\top)\E[(\theta_{k,r-1}-\theta_{\star})(\theta_{k,r-1}-\theta_{\star})^\top \vert X_k](\mathbb{I}_d-\alpha_kX_kX_k^\top)
	+4\alpha_k^2(\ell-1) X_k X_k^\top
	\\
 &\quad+\alpha_k^2\Big(\E[(X_k^\top(\theta_{\star}-\theta_{k,r-1}))^2\vert X_k]+1\Big)\big(2X_kX_k^\top+X_k^\top X_k\mathbb{I}_d\big).
\end{align*}
\end{proposition}
\begin{proof}
By definition and Lemma \ref{lemma: cov},
\begin{align*}
&\E[(\theta_{k,r}-\theta_{\star})(\theta_{k,r}-\theta_{\star})^\top \vert X_k]
\\
&=\E[(\theta_{k,r-1}-\theta_{\star})(\theta_{k,r-1}-\theta_{\star})^\top \vert X_k]
\\&
\quad+\alpha_k\E[((Y_k-X_k^\top\theta_{k,r-1})X_k^\top \xi_{k,r}\xi_{k,r})(\theta_{k,r-1}-\theta_{\star})^\top \vert X_k]
\\
&\quad+\alpha_k\E[(\theta_{k,r-1}-\theta_{\star})((Y_k-X_k^\top\theta_{k,r-1})X_k^\top \xi_{k,r}\xi_{k,r})^\top \vert X_k]
\\
&\quad+\alpha_k^2\E[((Y_k-X_k^\top\theta_{k,r-1})X_k^\top \xi_{k,r})^2\xi_{k,r}\xi_{k,r}^\top \vert X_k]
\\
&=\E[(\theta_{k,r-1}-\theta_{\star})(\theta_{k,r-1}-\theta_{\star})^\top \vert X_k]
\\&
\quad+\alpha_k\E[((Y_k-X_k^\top\theta_{k,r-1})X_k)(\theta_{k,r-1}-\theta_{\star})^\top \vert X_k]
\\
&\quad+\alpha_k\E[(\theta_{k,r-1}-\theta_{\star})((Y_k-X_k^\top\theta_{k,r-1})X_k)^\top \vert X_k]
\\
&\quad+\alpha_k^2\E[(Y_k-X_k^\top\theta_{k,r-1})^2\vert X_k](2X_kX_k^\top+X_k^\top X_k\mathbb{I}_d) .
\end{align*}
Since $X_k^\top(\theta_{\star}-\theta_{k,r-1})$ is a real number, we have $X_k^\top(\theta_{\star}-\theta_{k,r-1})X_k(\theta_{k,r-1}-\theta_{\star})^\top= X_k X_k^\top(\theta_{\star}-\theta_{k,r-1})(\theta_{k,r-1}-\theta_{\star})^\top.$ Using moreover that $\ep_k$ and $X_k$ are independent and Lemma \ref{lemma: ep} for the inequality in the last line, yields
\begin{align*}
&\E[((Y_k-X_k^\top\theta_{k,r-1})X_k)(\theta_{k,r-1}-\theta_{\star})^\top \vert X_k]
+\E[(\theta_{k,r-1}-\theta_{\star})((Y_k-X_k^\top\theta_{k,r-1})X_k \vert X_k]
	\\
	&=\E[X_k^\top(\theta_{\star}-\theta_{k,r-1})X_k(\theta_{k,r-1}-\theta_{\star})^\top \vert X_k]+\E[(\theta_{k,r-1}-\theta_{\star})(X_k^\top(\theta_{\star}-\theta_{k,r-1})X_k)^\top \vert X_k]
	\\&\quad+\E[\ep_kX_k(\theta_{k,r-1}-\theta_{\star})^\top \vert X_k]+\E[(\theta_{k,r-1}-\theta_{\star})\ep_kX_k^\top \vert X_k]
		\\
	&=-X_kX_k^\top\E[(\theta_{k,r-1}-\theta_{\star})(\theta_{k,r-1}-\theta_{\star})^\top \vert X_k]-\E[(\theta_{k,r-1}-\theta_{\star})(\theta_{k,r-1}-\theta_{\star})^\top \vert X_k]X_kX_k^\top
	\\&\quad+X_k\E[\ep_k\theta_{k,r-1}^\top \vert X_k]+\E[\ep_k \theta_{k,r-1}\vert X_k]X_k^\top \\
	&\leq_{\mathsf{L}} -X_kX_k^\top\E[(\theta_{k,r-1}-\theta_{\star})(\theta_{k,r-1}-\theta_{\star})^\top \vert X_k]-\E[(\theta_{k,r-1}-\theta_{\star})(\theta_{k,r-1}-\theta_{\star})^\top \vert X_k]X_kX_k^\top
	\\&\quad+4\alpha_k (\ell-1) X_k X_k^\top.
\end{align*}
Arguing similarly and using Lemma \ref{lemma: ep} for the last inequality,
\begin{align*}
&\E[(Y_k-X_k^\top\theta_{k,r-1})^2\vert X_k]
\\
&=\E[(X_k^\top(\theta_{\star}-\theta_{k,r-1}))^2\vert X_k]+\E[\ep_k^2\vert X_k]
+2\E[\ep_kX_k^\top(\theta_{\star}-\theta_{k,r-1})\vert X_k]
\\&=\E[(X_k^\top(\theta_{\star}-\theta_{k,r-1}))^2\vert X_k]+1
-2X_k^\top\E[\ep_k\theta_{k,r-1}\vert X_k]\\
&\leq \E[(X_k^\top(\theta_{\star}-\theta_{k,r-1}))^2\vert X_k]+1.
\end{align*}
By combining the previous results, we obtain
\begin{align*}
	&\E[(\theta_{k,r}-\theta_{\star})(\theta_{k,r}-\theta_{\star})^\top \vert X_k]
\\
	&\leq_{\mathsf{L}}\E[(\theta_{k,r-1}-\theta_{\star})(\theta_{k,r-1}-\theta_{\star})^\top \vert X_k]
	\\&
	\quad-\alpha_k\big(X_kX_k^\top\E[(\theta_{k,r-1}-\theta_{\star})(\theta_{k,r-1}-\theta_{\star})^\top \vert X_k]+\E[(\theta_{k,r-1}-\theta_{\star})(\theta_{k,r-1}-\theta_{\star})^\top \vert X_k]X_kX_k^\top\big)
	\\&\quad+4\alpha_k^2 (\ell-1) X_k X_k^\top
	\\
	&\quad+\alpha_k^2(\E[(X_k^\top(\theta_{\star}-\theta_{k,r-1}))^2\vert X_k]+1
	)(2X_kX_k^\top+X_k^\top X_k\mathbb{I}_d) 
	\\
	&=(\mathbb{I}_d-\alpha_kX_kX_k^\top)\E[(\theta_{k,r-1}-\theta_{\star})(\theta_{k,r-1}-\theta_{\star})^\top \vert X_k](\mathbb{I}_d-\alpha_kX_kX_k^\top)
	\\&\quad -\alpha_k^2X_kX_k^\top\E[(\theta_{k,r-1}-\theta_{\star})(\theta_{k,r-1}-\theta_{\star})^\top \vert X_k]X_kX_k^\top
	\\&\quad+4\alpha_k^2 (\ell-1) X_k X_k^\top
	\\
	&\quad+\alpha_k^2(\E[(X_k^\top(\theta_{\star}-\theta_{k,r-1}))^2\vert X_k]+1
	)(2X_kX_k^\top+X_k^\top X_k\mathbb{I}_d) 
	\\&\leq_{\mathsf L}
	(\mathbb{I}_d-\alpha_kX_kX_k^\top)\E[(\theta_{k,r-1}-\theta_{\star})(\theta_{k,r-1}-\theta_{\star})^\top \vert X_k](\mathbb{I}_d-\alpha_kX_kX_k^\top)
	\\&\quad+4\alpha_k^2 (\ell-1) X_k X_k^\top
	\\
	&\quad+\alpha_k^2(\E[(X_k^\top(\theta_{\star}-\theta_{k,r-1}))^2\vert X_k]+1
	)(2X_kX_k^\top+X_k^\top X_k\mathbb{I}_d),
\end{align*}
completing the proof.
\end{proof}

Additionally to the previous lemma, Proposition \ref{prop: rec 1} also requires us to bound the MSPE of $\theta_{k,r}$ but not weighted by the independent datum $X,$ but by the used training datum $X_k$. For doing so, Proposition \ref{prop: rec 1} itself will prove to be useful. 
\begin{lemma}\label{lemma: cond mse}
 Let $k\in\N$ and $r=1,\ldots,\ell.$ If $\alpha_k\leq (4 X_k^\top X_k)^{-1}$ it holds \begin{align*}
	&\E[(X_k^\top(\theta_{k,r}-\theta_{\star}))^2 \vert X_k]
    \leq (1-\alpha_kX_k^\top X_k)^r\E[(X_k^\top(\theta_{k-1,\ell}-\theta_{\star}))^2 \vert X_k]
	+4r\ell\alpha_k^2 (X_k^\top X_k)^2.
\end{align*}
\end{lemma}
\begin{proof}
Applying Proposition \ref{prop: rec 1} and using $X_k^\top(\mathbb{I}_d-\alpha_kX_kX_k^\top)=(1-\alpha_kX_k^\top X_k)X_k^\top,$ $(1-\alpha_kX_k^\top X_k)\in \mathbb{R},$ and $\alpha_kX_k^\top X_k\leq 1/4,$ gives
\begin{align*}
	&\E[(X_k^\top(\theta_{k,r}-\theta_{\star}))^2 \vert X_k]
    \\
    &=X_k^\top\E[(\theta_{k,r}-\theta_{\star})(\theta_{k,r}-\theta_{\star})^\top \vert X_k]X_k
	\\
	&\leq
	X_k^\top(\mathbb{I}_d-\alpha_kX_kX_k^\top)\E[(\theta_{k,r-1}-\theta_{\star})(\theta_{k,r-1}-\theta_{\star})^\top \vert X_k](\mathbb{I}_d-\alpha_kX_kX_k^\top)X_k
	+4\alpha_k^2(\ell-1) (X_k^\top X_k)^2 
	\\
 &\quad+\alpha_k^2\Big(\E[(X_k^\top(\theta_{\star}-\theta_{k,r-1}))^2\vert X_k]+1\Big)X_k^\top\big(2X_kX_k^\top+X_k^\top X_k\mathbb{I}_d\big)X_k
 \\
 &=	(1-\alpha_kX_k^\top X_k)^2\E[(X_k^\top(\theta_{k,r-1}-\theta_{\star}))^2 \vert X_k]
	+4\alpha_k^2(\ell-1) (X_k^\top X_k)^2 
	\\
 &\quad+3\alpha_k^2 (X_k^\top X_k)^2\Big(\E[(X_k^\top(\theta_{\star}-\theta_{k,r-1}))^2\vert X_k]+1\Big)
  \\
 &=	(1-2\alpha_kX_k^\top X_k+4\alpha_k^2(X_k^\top X_k)^2)\E[(X_k^\top(\theta_{k,r-1}-\theta_{\star}))^2 \vert X_k]
	+4\alpha_k^2(\ell-1) (X_k^\top X_k)^2 
	\\
 &\quad+3\alpha_k^2 (X_k^\top X_k)^2
 \\
 &\leq 	(1-\alpha_kX_k^\top X_k)\E[(X_k^\top(\theta_{k,r-1}-\theta_{\star}))^2 \vert X_k]
	+4\alpha_k^2\ell (X_k^\top X_k)^2.
\end{align*}
Applying this inequality recursively now gives
\begin{align*}
	&\E[(X_k^\top(\theta_{k,r}-\theta_{\star}))^2 \vert X_k]
 \\
 &\leq 	(1-\alpha_kX_k^\top X_k)^r\E[(X_k^\top(\theta_{k-1,\ell}-\theta_{\star}))^2 \vert X_k]
	+4\alpha_k^2\ell (X_k^\top X_k)^2\sum_{i=0}^{r-1}(1-\alpha_kX_k^\top X_k)^i
    \\
    &\leq (1-\alpha_kX_k^\top X_k)^r\E[(X_k^\top(\theta_{k-1,\ell}-\theta_{\star}))^2 \vert X_k]
	+4r\ell\alpha_k^2 (X_k^\top X_k)^2,
\end{align*}
where we used that $\alpha_k X_k^\top X_k\leq 1/4$ guarantees $\vert 1-\alpha_kX_k^\top X_k\vert \leq 1$. 
This concludes the proof.
\end{proof}
Combining the previous lemma with Proposition \ref{prop: rec 1} now allows us to derive an upper bound for the MSPE of $\theta_{k,\ell}$ in terms of $\theta_{k-1,\ell},$ which subsequently will lead to the main result of this paper in Theorem \ref{thm: mspe}.
\begin{proposition}\label{prop: rec mspe}
Let $k\in\N.$ If $X_k^\top X_k\leq b$ almost surely and $\alpha_k\leq (4\ell b)^{-1},$ then
\begin{align*}
	\operatorname{MSPE}(\theta_{k,\ell})
	&\leq
\Big(1-2\alpha_k\ell\lambda_{\min\neq 0}(\Sigma)+3\alpha_k^2b\ell\big(2\ell\lambda_{\max}(\Sigma)+\tr(\Sigma)\big)\Big)\operatorname{MSPE}(\theta_{k-1,\ell})
\\
&\quad +4\alpha_k^2b\ell\bigg(\ell\lambda_{\max}(\Sigma)
+\frac{15}{16}\tr(\Sigma)\bigg).
\end{align*}
\end{proposition}
\begin{proof}
Proposition \ref{prop: rec 1} gives
\begin{align*}
	&\E[(\theta_{k,r}-\theta_{\star})(\theta_{k,r}-\theta_{\star})^\top \vert X_k]
	\\
	&\leq_{\mathsf L}
	(\mathbb{I}_d-\alpha_kX_kX_k^\top)\E[(\theta_{k,r-1}-\theta_{\star})(\theta_{k,r-1}-\theta_{\star})^\top \vert X_k](\mathbb{I}_d-\alpha_kX_kX_k^\top)
	+4\alpha_k^2(\ell-1) X_k X_k^\top
	\\
 &\quad+\alpha_k^2\Big(\E[(X_k^\top(\theta_{\star}-\theta_{k,r-1}))^2\vert X_k]+1\Big)\big(2X_kX_k^\top+X_k^\top X_k\mathbb{I}_d\big)
 \end{align*}
 Using Lemma \ref{lemma: cond mse}, $\alpha_k\leq (4\ell b)^{-1},$ and $X_kX_k^\top \leq_{\mathsf L} X_k^\top X_k \mathbb{I}_d,$
 \begin{align*}
        &\Big(\E[(X_k^\top(\theta_{\star}-\theta_{k,r-1}))^2\vert X_k]+1\Big)\big(2X_kX_k^\top+X_k^\top X_k\mathbb{I}_d\big) \\
		&\leq_{\mathsf L}3X_k^\top X_k\mathbb{I}_d \bigg( (1-\alpha_k X_k^\top X_k)^{r-1}\E[(X_k^\top(\theta_{\star}-\theta_{k-1,\ell}))^2\vert X_k]+4(r-1)\ell\alpha_k^2(X_k^\top X_k)^2+1\bigg)
	\\
 &\leq_{\mathsf L}3X_k^\top X_k\mathbb{I}_d \bigg( (1-\alpha_k X_k^\top X_k)^{r-1}\E[(X_k^\top(\theta_{\star}-\theta_{k-1,\ell}))^2\vert X_k]+\frac 54\bigg).
	\end{align*}
	Combining the previous inequalities and applying the combined inequality recursively using moreover
    \begin{align*}
        \sum_{i=0}^{\ell-1}(\mathbb{I}_d-\alpha_kX_kX_k^\top)^i X_kX_k^\top(\mathbb{I}_d-\alpha_kX_kX_k^\top)^i
        =X_kX_k^\top \sum_{i=0}^{\ell-1}(1-\alpha_kX_k^\top X_k)^{2i}
        \leq \ell X_kX_k^\top
    \end{align*}
    yields
	\begin{align*}
	&\E[(\theta_{k,\ell}-\theta_{\star})(\theta_{k,\ell}-\theta_{\star})^\top\vert X_k]
	\\
	&\leq_{\mathsf L} 		(\mathbb{I}_d-\alpha_kX_kX_k^\top)^\ell\E[(\theta_{k-1,\ell}-\theta_{\star})(\theta_{k-1,\ell}-\theta_{\star})^\top](\mathbb{I}_d-\alpha_kX_kX_k^\top)^\ell
	\\&\quad +4\alpha_k^2(\ell-1)\sum_{i=0}^{\ell-1}(\mathbb{I}_d-\alpha_kX_kX_k^\top)^i X_kX_k^\top(\mathbb{I}_d-\alpha_kX_kX_k^\top)^i
	\\
	&\quad+3\alpha_k^2X_k^\top X_k\sum_{i=0}^{\ell-1}\bigg(\big( (1-\alpha_k X_k^\top X_k)^{\ell-1-i}\E[(X_k^\top(\theta_{\star}-\theta_{k-1,\ell}))^2\vert X_k]
	+\frac 54\big)
(\mathbb{I}_d-\alpha_kX_kX_k^\top)^{2i}\bigg)
	\\
&\leq_{\mathsf L}	(\mathbb{I}_d-\alpha_kX_kX_k^\top)^\ell\E[(\theta_{k-1,\ell}-\theta_{\star})(\theta_{k-1,\ell}-\theta_{\star})^\top](\mathbb{I}_d-\alpha_kX_kX_k^\top)^\ell
\\&\quad +4\alpha_k^2(\ell-1)\ell X_kX_k^\top
\\
&\quad+3\alpha_k^2X_k^\top X_k\sum_{i=0}^{\ell-1}\bigg(\big( (1-\alpha_kX_k^\top X_k)^{\ell-1-i}\E[(X_k^\top(\theta_{\star}-\theta_{k-1,\ell}))^2\vert X_k]
+\frac 54\big)
(\mathbb{I}_d-\alpha_kX_kX_k^\top)^{2i}\bigg).
\end{align*}
The positive semi-definite matrix $X_kX_k^\top$ is rank one with eigenvector $\|X_k\|^{-1} X_k$ corresponding to the only non-zero eigenvalue $\|X_k\|^2.$ Since also $\alpha_k \leq 1/\|X_k\|,$ $\mathbb{I}_d-\alpha_kX_kX_k^\top$ is positive semi-definite and all eigenvalues lie in $[0,1].$ Thus, $(\mathbb{I}_d-\alpha_kX_kX_k^\top)^n$ is also positive semi-definite and all eigenvalues lie in $[0,1].$ Hence, $(\mathbb{I}_d-\alpha_kX_kX_k^\top)^n \leq_{\mathsf L}\mathbb{I}_d$ and thus,
	\begin{align*}
	&\E[(\theta_{k,\ell}-\theta_{\star})(\theta_{k,\ell}-\theta_{\star})^\top\vert X_k]
	\\
	&\leq_{\mathsf L} 		(\mathbb{I}_d-\frac{1-(1-\alpha_kX_k^\top X_k)^{\ell}}{X_k^\top X_k}X_kX_k^\top)\E[(\theta_{k-1,\ell}-\theta_{\star})(\theta_{k-1,\ell}-\theta_{\star})^\top](\mathbb{I}_d-\frac{1-(1-\alpha_kX_k^\top X_k)^{\ell}}{X_k^\top X_k}X_kX_k^\top)
\\&\quad +4\alpha_k^2(\ell-1)\ell X_kX_k^\top
+3\alpha_k^2X_k^\top X_k\mathbb{I}_d\sum_{i=0}^{\ell-1}\big( (1-\alpha_kX_k^\top X_k)^{\ell-1-i}\E[(X_k^\top(\theta_{\star}-\theta_{k-1,\ell}))^2\vert X_k]
+\frac 54\big)
		\\
&\leq_{\mathsf L}\E[(\theta_{k-1,\ell}-\theta_{\star})(\theta_{k-1,\ell}-\theta_{\star})^\top]
\\
&\quad-\frac{1-(1-\alpha_kX_k^\top X_k)^{\ell}}{X_k^\top X_k}\Big(X_kX_k^\top\E[(\theta_{k-1,\ell}-\theta_{\star})(\theta_{k-1,\ell}-\theta_{\star})^\top]\\&\quad+\E[(\theta_{k-1,\ell}-\theta_{\star})(\theta_{k-1,\ell}-\theta_{\star})^\top]X_kX_k^\top \Big)
\\
&\quad +\Big(\frac{1-(1-\alpha_kX_k^\top X_k)^{\ell}}{X_k^\top X_k}\Big)^2X_k^\top\E[(\theta_{k-1,\ell}-\theta_{\star})(\theta_{k-1,\ell}-\theta_{\star})^\top]X_k X_kX_k^\top
\\&\quad +4\alpha_k^2(\ell-1)\ell X_kX_k^\top
+3\alpha_k^2X_k^\top X_k\ell\mathbb{I}_d\big( \E[(X_k^\top(\theta_{\star}-\theta_{k-1,\ell}))^2\vert X_k]
+\frac 54\big)
		\\
&\leq_{\mathsf L}\E[(\theta_{k-1,\ell}-\theta_{\star})(\theta_{k-1,\ell}-\theta_{\star})^\top]
\\&\quad-\frac{1-(1-\alpha_kX_k^\top X_k)^{\ell}}{X_k^\top X_k}(X_kX_k^\top\E[(\theta_{k-1,\ell}-\theta_{\star})\Big(\theta_{k-1,\ell}-\theta_{\star})^\top]
\\&\quad+\E[(\theta_{k-1,\ell}-\theta_{\star})(\theta_{k-1,\ell}-\theta_{\star})^\top]X_kX_k^\top \Big)
\\
&\quad +4\alpha_k^2\ell^2( X_k^\top\E[(\theta_{k-1,\ell}-\theta_{\star})(\theta_{k-1,\ell}-\theta_{\star})^\top]X_k +1)X_kX_k^\top
\\
&\quad
+3\alpha_k^2X_k^\top X_k\ell\mathbb{I}_d\big( X_k^\top\E[(\theta_{k-1,\ell}-\theta_{\star})(\theta_{k-1,\ell}-\theta_{\star})^\top]X_k
+\frac 54\big),
\end{align*}
where we argued as in \eqref{eq: 1} in the last step. Now let $X$ be an independent draw from the input distribution. The independence of $X,X_k$ and $\theta_{k-1,\ell}$ gives
\begin{align*}
&\operatorname{MSPE}(\theta_{k,\ell})
\\
&=\E[X^\top\E[(\theta_{k,\ell}-\theta_{\star})(\theta_{k,\ell}-\theta_{\star})^\top] X]
\\
&\leq \operatorname{MSPE}(\theta_{k-1,\ell})
\\&\quad-2\E\Big[\frac{1-(1-\alpha_kX_k^\top X_k)^{\ell}}{X_k^\top X_k}X^\top X_kX_k^\top\E[(\theta_{k-1,\ell}-\theta_{\star})(\theta_{k-1,\ell}-\theta_{\star})^\top]X\Big]
\\
&\quad +4\alpha_k^2\ell^2 \E[(X_k^\top\E[(\theta_{k-1,\ell}-\theta_{\star})(\theta_{k-1,\ell}-\theta_{\star})^\top]X_k+1) (X^\top X_k)^2]
\\
&\quad+3\alpha_k^2\ell\E\big[X^\top XX_k^\top X_k\big( X_k^\top\E[(\theta_{k-1,\ell}-\theta_{\star})(\theta_{k-1,\ell}-\theta_{\star})^\top]X_k
+\frac 54\big)\big]
\\
&= \operatorname{MSPE}(\theta_{k-1,\ell})
\\&\quad-2\E\Big[\frac{1-(1-\alpha_kX_k^\top X_k)^{\ell}}{X_k^\top X_k}X^\top X_kX_k^\top\E[(\theta_{k-1,\ell}-\theta_{\star})(\theta_{k-1,\ell}-\theta_{\star})^\top]X\Big]
\\
&\quad +4\alpha_k^2\ell^2 \E[(X_k^\top\E[(\theta_{k-1,\ell}-\theta_{\star})(\theta_{k-1,\ell}-\theta_{\star})^\top]X_k+1) X_k^\top\Sigma X_k]
\\
&\quad+3\alpha_k^2\tr(\Sigma)\ell\E\big[X_k^\top X_k\big( X_k^\top\E[(\theta_{k-1,\ell}-\theta_{\star})(\theta_{k-1,\ell}-\theta_{\star})^\top]X_k
+\frac 54\big)\big]
\\
&\leq \operatorname{MSPE}(\theta_{k-1,\ell})
\\&\quad-2\E\Big[\frac{1-(1-\alpha_kX_k^\top X_k)^{\ell}}{X_k^\top X_k}X^\top X_kX_k^\top\E[(\theta_{k-1,\ell}-\theta_{\star})(\theta_{k-1,\ell}-\theta_{\star})^\top]X\Big]
\\
&\quad +4\alpha_k^2\ell^2\lambda_{\max}(\Sigma)b \big(\E[X_k^\top\E[(\theta_{k-1,\ell}-\theta_{\star})(\theta_{k-1,\ell}-\theta_{\star})^\top]X_k ]+1\big)
\\
&\quad+3\alpha_k^2\tr(\Sigma)b\ell\big(\E[ X_k^\top\E[(\theta_{k-1,\ell}-\theta_{\star})(\theta_{k-1,\ell}-\theta_{\star})^\top]X_k
]+\frac 54\big)
\\
&
=\bigg(1+4\alpha_k^2b\ell\bigg(\ell\lambda_{\max}(\Sigma)+\frac{3}{4}\tr(\Sigma)\bigg)\bigg)\operatorname{MSPE}(\theta_{k-1,\ell})
\\&\quad-2\E\Big[\frac{1-(1-\alpha_kX_k^\top X_k)^{\ell}}{X_k^\top X_k}X^\top X_kX_k^\top\E[(\theta_{k-1,\ell}-\theta_{\star})(\theta_{k-1,\ell}-\theta_{\star})^\top]X\Big]
\\
&\quad +4\alpha_k^2b\ell\bigg(\ell\lambda_{\max}(\Sigma)
+\frac{15}{16}\tr(\Sigma)\bigg).
\end{align*}
Now applying Taylor's Theorem to the function $f(x)=(1-\alpha_k x)^\ell,$ gives for $x\in[0,b],$
\begin{align*}
\big\vert 1-\alpha_k\ell x-(1-\alpha_k x)^\ell\big\vert
&=\big\vert f(0)+f'(0)x-f(x)\big\vert
\\
&\leq x^2\sup_{x\in[0,b]} \Big\vert \frac{f''(x)}{2} \Big\vert
\\
&\leq \frac{x^2\alpha_k^2\ell^2}{2}.
\end{align*}
Thus
\begin{align*}
	&\operatorname{MSPE}(\theta_{k,\ell})
	\\
	&\leq\bigg(1+4\alpha_k^2b\ell\bigg(\ell\lambda_{\max}(\Sigma)+\frac{3}{4}\tr(\Sigma)\bigg)\bigg)\operatorname{MSPE}(\theta_{k-1,\ell})
	\\&\quad-2\E\Big[\frac{1-(1-\alpha_kX_k^\top X_k)^{\ell}}{X_k^\top X_k}X^\top X_kX_k^\top\E[(\theta_{k-1,\ell}-\theta_{\star})(\theta_{k-1,\ell}-\theta_{\star})^\top]X\Big]
	\\
	&\quad +4\alpha_k^2b\ell\bigg(\ell\lambda_{\max}(\Sigma)
	+\frac{15}{16}\tr(\Sigma)\bigg)-2\alpha_k\ell\E\Big[(\theta_{k-1,\ell}-\theta_{\star})^\top\Sigma X_kX_k^\top(\theta_{k-1,\ell}-\theta_{\star})\Big]
	\\&\quad+2\alpha_k\ell\E\Big[(\theta_{k-1,\ell}-\theta_{\star})^\top\Sigma X_kX_k^\top(\theta_{k-1,\ell}-\theta_{\star})\Big]
	\\
	&=\bigg(1+4\alpha_k^2b\ell\bigg(\ell\lambda_{\max}(\Sigma)+\frac{3}{4}\tr(\Sigma)\bigg)\bigg)\operatorname{MSPE}(\theta_{k-1,\ell})
	\\&\quad-2\E\Big[\frac{1-\alpha_k\ell X_k^\top X_k-(1-\alpha_kX_k^\top X_k)^{\ell}}{X_k^\top X_k}(\theta_{k-1,\ell}-\theta_{\star})^\top\Sigma X_kX_k^\top(\theta_{k-1,\ell}-\theta_{\star})\Big]
	\\
	&\quad +4\alpha_k^2b\ell\bigg(\ell\lambda_{\max}(\Sigma)
	+\frac{15}{16}\tr(\Sigma)\bigg)-2\alpha_k\ell\E\Big[(\theta_{k-1,\ell}-\theta_{\star})^\top\Sigma^2(\theta_{k-1,\ell}-\theta_{\star})\Big]
		\\
	&\leq
	\bigg(1-2\alpha_k\ell\lambda_{\min\neq 0}(\Sigma)+4\alpha_k^2b\ell\bigg(\ell\lambda_{\max}(\Sigma)+\frac{3}{4}\tr(\Sigma)\bigg)\bigg)\operatorname{MSPE}(\theta_{k-1,\ell})
	\\&\quad+2\alpha_k^2\ell^2\E\Big[X_k^\top X_k\vert (\theta_{k-1,\ell}-\theta_{\star})^\top\Sigma X_kX_k^\top(\theta_{k-1,\ell}-\theta_{\star})\vert\Big]
	\\
	&\quad +4\alpha_k^2b\ell\bigg(\ell\lambda_{\max}(\Sigma)
	+\frac{15}{16}\tr(\Sigma)\bigg)
		\\
	&\leq
	\bigg(1-2\alpha_k\ell\lambda_{\min\neq 0}(\Sigma)+4\alpha_k^2b\ell\bigg(\ell\lambda_{\max}(\Sigma)+\frac{3}{4}\tr(\Sigma)\bigg)\bigg)\operatorname{MSPE}(\theta_{k-1,\ell})
	\\&\quad+2\alpha_k^2\ell^2b\Big(\E[ ((\theta_{k-1,\ell}-\theta_{\star})^\top\Sigma X_k)^2]\E[(X_k^\top(\theta_{k-1,\ell}-\theta_{\star}))^2]\Big)^{1/2}
	\\
	&\quad +4\alpha_k^2b\ell\bigg(\ell\lambda_{\max}(\Sigma)
	+\frac{15}{16}\tr(\Sigma)\bigg)
			\\
	&=
	\bigg(1-2\alpha_k\ell\lambda_{\min\neq 0}(\Sigma)+4\alpha_k^2b\ell\bigg(\ell\lambda_{\max}(\Sigma)+\frac{3}{4}\tr(\Sigma)\bigg)\bigg)\operatorname{MSPE}(\theta_{k-1,\ell})
	\\&\quad+2\alpha_k^2\ell^2b\Big(\E[ (\theta_{k-1,\ell}-\theta_{\star})^\top\Sigma^3(\theta_{k-1,\ell}-\theta_{\star})]\operatorname{MSPE}(\theta_{k-1,\ell})\Big)^{1/2}
	\\
	&\quad +4\alpha_k^2b\ell\bigg(\ell\lambda_{\max}(\Sigma)
	+\frac{15}{16}\tr(\Sigma)\bigg)
				\\
	&\leq
	\big(1-2\alpha_k\ell\lambda_{\min\neq 0}(\Sigma)+3\alpha_k^2b\ell(2\ell\lambda_{\max}(\Sigma)+\tr(\Sigma))\big)\operatorname{MSPE}(\theta_{k-1,\ell})
	\\
	&\quad +4\alpha_k^2b\ell\bigg(\ell\lambda_{\max}(\Sigma)
	+\frac{15}{16}\tr(\Sigma)\bigg),
\end{align*}
where we used that $\Sigma^2\geq_{\mathsf L}\lambda_{\min\neq 0}(\Sigma)\Sigma$ and $\Sigma^3\leq_{\mathsf L}\lambda_{\max}(\Sigma)^2\Sigma$.
This completes the proof.
\end{proof}
Applying Proposition \ref{prop: rec mspe} inductively now almost immediately proves Theorem \ref{thm: mspe}.
\begin{proof}[Proof of Theorem \ref{thm: mspe}]
Note first that our assumptions guarantee $\alpha_i\leq (4\ell b)^{-1}$ for all $i\in[k]$ and additionally
\[\big(1-2\alpha_i\ell\lambda_{\min\neq 0}(\Sigma)+3\alpha_i^2b\ell(2\ell\lambda_{\max}(\Sigma)+\tr(\Sigma))\big)
\leq (1-\alpha_i\ell\lambda_{\min\neq 0}(\Sigma)),\quad i\in[k]. \] Thus, we can apply Proposition \ref{prop: rec mspe} and obtain by arguing inductively
\begin{align*}
	&\operatorname{MSPE}(\theta_{k,\ell})
\\
&\leq
\big(1-\alpha_k\ell\lambda_{\min\neq 0}(\Sigma)\big)\operatorname{MSPE}(\theta_{k-1,\ell})
+4\alpha_k^2b\ell\bigg(\ell\lambda_{\max}(\Sigma)
+\frac{15}{16}\tr(\Sigma)\bigg)
\\
&\leq\operatorname{MSPE}(\theta_{0,\ell}) \prod_{i=1}^{k}\big(1-\alpha_i\ell\lambda_{\min\neq 0}(\Sigma)\big)
\\
&\quad+4b\ell\bigg(\ell\lambda_{\max}(\Sigma)
+\frac{15}{16}\tr(\Sigma)\bigg)\sum_{i=1}^{k}\alpha_i^2\prod_{j=i+1}^{k}\big(1-\alpha_j\ell\lambda_{\min\neq 0}(\Sigma)\big)
\\
&
\leq
\exp\bigg(-\ell\lambda_{\min\neq 0}(\Sigma)\sum_{i=1}^{k}\alpha_i\bigg)\operatorname{MSPE}(\theta_{0,\ell}) 
\\
&\quad+4b\ell\bigg(\ell\lambda_{\max}(\Sigma)
+\frac{15}{16}\tr(\Sigma)\bigg)\sum_{i=1}^{k}\alpha_i^2\exp\bigg(-\ell\lambda_{\min\neq 0}(\Sigma)\sum_{j=i+1}^{k}\alpha_j\bigg).
\end{align*}
Now, denoting $c=c_1c_2$ it holds
\[\alpha_i\geq\frac{2}{\ell\lambda_{\min\neq0}(\Sigma)(c+i)}, \]
and arguing as in the derivation of equation (4.11) in \cite{bos24} gives for $k_1\leq k_2\in[k]$ 
\begin{align*}
	&\exp\left(-\lambda_{\min\neq 0}(\Sigma)\ell\sum_{i=k_1}^{k_2}\alpha_i\right)\leq \left(\frac{k_1+c}{k_2+1+c}\right)^2,
\end{align*}
leading to 
\begin{align*}
	\operatorname{MSPE}(\theta_{k,\ell})
&
\leq 
\left(\frac{1+c}{k+1+c}\right)^2\operatorname{MSPE}(\theta_{0,\ell}) 
+4b\ell\bigg(\ell\lambda_{\max}(\Sigma)
+\frac{15}{16}\tr(\Sigma)\bigg)\sum_{i=1}^{k}\alpha_i^2\left(\frac{i+1+c}{k+1+c}\right)^2
\\
&
=
\left(\frac{1+c}{k+1+c}\right)^2\operatorname{MSPE}(\theta_{0,\ell}) 
+4bc_1^2\bigg(\lambda_{\max}(\Sigma)
+\frac{15}{16\ell}\tr(\Sigma)\bigg)\frac{1}{(k+1+c)^2}\sum_{i=1}^{k}\left(\frac{i+1+c}{i+c}\right)^2
\\
&
\leq 
\left(\frac{1+c}{k+1+c}\right)^2\operatorname{MSPE}(\theta_{0,\ell}) 
+16bc_1^2\bigg(\lambda_{\max}(\Sigma)
+\frac{\tr(\Sigma)}{\ell}\bigg)\frac{k}{(k+1+c)^2},
\end{align*}
where we used the definition of $\alpha_i$ in the second to last step.
This completes the proof.
\end{proof}
\section{Analysis of Adjusted Forward Gradient Descent}\label{app: afgd}

\begin{lemma}
\label{lem.moments}
For a $d$-dimensional vector $a$ and  $\xi\sim \mathcal{N}(0,\mathbb{I}_d),$ $\Vert a\Vert\E[\sign(a^\top \xi) \xi]= \sqrt{2/\pi}a.$ 
\end{lemma}

\begin{proof}
It is sufficient to consider vectors $a$ with $\|a\|=1.$ Then, $\Cov(a^\top\xi, \xi-a^\top\xi a)= \E[a^\top\xi (\xi-a^\top\xi a)^\top]= \E[a^\top \xi \xi^\top]-\E[(a^\top \xi)^2 a^\top]= a^\top - a^\top=0.$ Since $(a^\top\xi, \xi-a^\top\xi a)$ follows a multivariate normal distribution, $a^\top\xi$ and $\xi-a^\top\xi a$ are independent. If $Z\sim \mathcal{N}(0,1),$ $E[|Z|]=\sqrt{2/\pi}.$ Since $\xi =a^\top \xi a + (\xi-a^\top\xi a)$ and $a^\top \xi\sim \mathcal{N}(0,1),$ $\E[\xi \sign(a^\top \xi)]=\E[|a^\top \xi|] a+ 0= \sqrt{2/\pi}a.$
\end{proof}
To avoid confusion, we denote the aFGD($\ell$) iterates in the following by $\omega,$ that is,
$$\omega_{k,r}=\omega_{k,r-1}+\beta_k(Y_k-X_k^\top\omega_{k,r-1})\|X_k\| \sign(X_k^\top \xi_{k,r})\xi_{k,r},\quad r\in[\ell]. $$
Then it holds
\begin{align*}
 &\omega_{k,r}-\theta_\star
 \\
 &=\omega_{k,r-1}-\theta_\star+\beta_k(Y_k-X_k^\top\omega_{k,r-1})\|X_k\| \sign(X_k^\top \xi_{k,r})\xi_{k,r}   
 \\
 &=\omega_{k,r-1}-\theta_\star+\beta_kX_k^\top(\theta_\star-\omega_{k,r-1})\|X_k\| \sign(X_k^\top \xi_{k,r})\xi_{k,r} 
 +\beta_k\ep_k\|X_k\| \sign(X_k^\top \xi_{k,r})\xi_{k,r}  
 \\
 &=\big(\mathbb{I}_d-\beta_k\|X_k\| \sign(X_k^\top \xi_{k,r})\xi_{k,r}X_k^\top\big)(\omega_{k,r-1}-\theta_\star)
 +\beta_k\ep_k\|X_k\| \sign(X_k^\top \xi_{k,r})\xi_{k,r}  
\end{align*}
By Lemma \ref{lem.moments}, we have $\E[ \sign(X_k^\top \xi_{k,r})\xi_{k,r}\vert X_k]=\sqrt{2/\pi} X_k.$ Since the noise $\ep_k$ is centered and independent of $(X_k,\xi_{k,r}),$ the second term has vanishing expectation and hence,
\begin{align*}
 \E[ \omega_{k,r}-\theta_{\star}\vert X_k]
 &= \Big(\mathbb{I}_d- \sqrt{\frac{2}{\pi}} \beta_k\|X_k\| \sign(X_k^\top \xi_{k,r})\xi_{k,r}X_k^\top\Big) \E[ \omega_{k,r-1}-\theta_{\star}\vert X_k],
\end{align*}
which leads to the same bias bound as FGD($\ell$) (see Theorem \ref{thm: bias}). Additionally, we can show the following result on the MSE of aFGD($\ell$). An analogous result for the MSPE as in Theorem \ref{thm: mspe} is also derivable, however due to the similarities in the proofs we here only investigate the simpler MSE.
\begin{theorem}\label{thm: afgd mspe}
	Let $k\in\N,b>0,$ and assume $\Vert X_1\Vert^2\leq b$ and $\lambda_{\min}(\Sigma)>0$. If the learning rate is of the form
	\[\beta_i=\frac{c_1}{\ell(c_1c_2+i)},\quad i=1,\ldots,k, \]
	for constants $c_1,c_2>0$ satisfying
	\[c_1\geq \frac{\sqrt{32\pi}}{\lambda_{\min}(\Sigma)},\quad c_2\geq \sqrt{\frac \pi 2}b\Big(1\lor \frac{4d\kappa(\Sigma)}{\ell}\Big),  \]
	then,
	\begin{align*}
		\operatorname{MSE}(\omega_{k,\ell})
		\leq\left(\frac{1+c_1c_2}{k+1+c_1c_2}\right)^2
\operatorname{MSE}(\omega_{0})
+8c_1^2\tr(\Sigma)\Big(\frac{2}{\pi}+\frac{d}{\ell}\Big)\frac{k}{(k+1+c_1c_2)^2}.
	\end{align*}
\end{theorem}
Hence for $\ell\geq d,$ aFGD($\ell$) also achieves the minimax optimal rate $d/k$.
\begin{proof}[Proof of Theorem \ref{thm: afgd mspe}]
In order to minimize redundancies, we only provide the main steps of the proof.
The definition of aFGD implies
\begin{align*}
 &   (\omega_{k,r}-\theta_\star)(\omega_{k,r}-\theta_\star)^\top
  \\
 &=\Big(\mathbb{I}_d-\beta_k\|X_k\| \sign(X_k^\top \xi_{k,r})\xi_{k,r}X_k^\top\Big)(\omega_{k,r-1}-\theta_\star)(\omega_{k,r-1}-\theta_\star)^\top
 \\&\quad\cdot\Big(\mathbb{I}_d-\beta_k\|X_k\| \sign(X_k^\top \xi_{k,r})\xi_{k,r}X_k^\top\Big)^\top
 \\
 &\quad+\beta_k\ep_k\|X_k\|\sign(X_k^\top \xi_{k,r}) \Big((\omega_{k,r-1}-\theta_\star) \xi_{k,r}^\top+\xi_{k,r}(\omega_{k,r-1}-\theta_\star)^\top\Big)
 \\
 &\quad-2\beta_k^2\ep_k\|X_k\|^2X_k^\top(\omega_{k,r-1}-\theta_\star)\xi_{k,r}\xi_{k,r}^\top
 +\beta_k^2\ep_k^2\|X_k\|^2 \xi_{k,r}\xi_{k,r}^\top.
\end{align*}
Arguing as in the proof of Lemma \ref{lemma: ep}, one can derive
\begin{align*}
  &\E[\ep_k\omega_{k,r}\vert X_k]
   =  \big(\mathbb{I}_d-\sqrt{2/\pi}\beta_kX_kX_k^\top
\big)\E[\ep_k\omega_{k,r-1}
 \vert X_k]
 + \sqrt{2/\pi}\beta_k X_k.
\end{align*}
Combined with the independence of $\ep_k$ and $\omega_{k-1,\ell},$
\begin{align*}
  &\E[\ep_k\omega_{k,r}\vert X_k]
   =  \sqrt{2/\pi}\beta_k \sum_{i=0}^{r-1}\big(1-\sqrt{2/\pi}\beta_kX_k^\top X_k
\big)^iX_k.
\end{align*}
Hence, since $\sqrt{2/\pi}\beta_k X_k^\top X_k\leq 1$
\begin{align*}
&\E\Big[\ep_k(\omega_{k,r-1}-\theta_\star)\big(\|X_k\| \sign(X_k^\top \xi_{k,r})\xi_{k,r}\big)^\top\Big \vert X_k\Big]
\leq_{\mathsf{L}} \frac 2{\pi}r\beta_k X_kX_k^\top, 
\end{align*}
and
\begin{align*}
    &\E[\ep_k\|X_k\|^2 \xi_{k,r}X_k^\top(\omega_{k,r-1}-\theta_\star)\xi_{k,r}^\top\vert X_k]
\geq_{\mathsf{L}}0.
\end{align*}
Additionally it holds
\begin{align*}
    \E[\ep_k^2\|X_k\|^2 \xi_{k,r}\xi_{k,r}^\top\vert X_k]=\Vert X_k\Vert^2\mathbb{I}_d,
\end{align*}
and thus
\begin{align}
\begin{split}\label{eq: afgd mse 1}
  &   \E[(\omega_{k,r}-\theta_\star)(\omega_{k,r}-\theta_\star)^\top\vert X_k]
 \\
 &\leq_{\mathsf{L}}\E\Big[\big(\mathbb{I}_d-\beta_k\|X_k\| \sign(X_k^\top \xi_{k,r})\xi_{k,r}X_k^\top\big)(\omega_{k,r-1}-\theta_\star)\\&\quad\cdot(\omega_{k,r-1}-\theta_\star)^\top\big(\mathbb{I}_d-\beta_k\|X_k\| \sign(X_k^\top \xi_{k,r})\xi_{k,r}X_k^\top\big)^\top\vert X_k\Big]
 \\
 &\quad+\beta_k^2\Big(\frac{4}{\pi}r X_kX_k^\top
 +\Vert X_k\Vert^2\mathbb{I}_d\Big)
 \\
 &=
  \E\Big[(\omega_{k,r-1}-\theta_\star)(\omega_{k,r-1}-\theta_\star)^\top\vert X_k\Big]-\sqrt{2/\pi}\beta_kX_kX_k^\top\E\Big[(\omega_{k,r-1}-\theta_\star)(\omega_{k,r-1}-\theta_\star)^\top\vert X_k\Big]
 \\
 &\quad-\sqrt{2/\pi}\beta_k\E\Big[(\omega_{k,r-1}-\theta_\star)(\omega_{k,r-1}-\theta_\star)^\top \vert X_k\Big]X_kX_k^\top
 \\&\quad+\beta_k^2\Vert X_k\Vert^2\E[X_k^\top(\omega_{k,r-1}-\theta_\star)(\omega_{k,r-1}-\theta_\star)^\top X_k \vert X_k]\mathbb{I}_d
+\beta_k^2\Big(\frac 4{\pi}r X_kX_k^\top
 +\Vert X_k\Vert^2\mathbb{I}_d\Big).  
\end{split}
\end{align}
Multiplying $X_k^\top$ from the left and $X_k$ from the right gives
\begin{align*}
   &\E[X_k^\top(\omega_{k,r}-\theta_\star)(\omega_{k,r}-\theta_\star)^\top X_k \vert X_k] 
    \\
   &\leq \Big(1-\sqrt{\frac{8}{\pi}}\beta_k \Vert X_k\Vert^2+\beta_k^2\Vert X_k\Vert^4\Big) \E\Big[X_k^\top(\omega_{k,r-1}-\theta_\star)(\omega_{k,r-1}-\theta_\star)^\top X_k\vert X_k\Big]
+\frac 8{\pi}r\beta_k^2\Vert X_k\Vert^4.
\end{align*}
Combined with $\beta_k\Vert X_k\Vert^2\leq \sqrt{2/\pi},$ this implies
\begin{align*}
   &\E[X_k^\top(\omega_{k,r}-\theta_\star)(\omega_{k,r}-\theta_\star)^\top X_k \vert X_k] 
\\
   &\leq \Big(1-\sqrt{\frac 2{\pi}}\beta_k \Vert X_k\Vert^2\Big)^r X_k^\top\E\Big[(\omega_{k-1,\ell}-\theta_\star)(\omega_{k-1,\ell}-\theta_\star)^\top \Big]X_k
+\frac 4{\pi}\beta_k^2\Vert X_k\Vert^4r(r+1).
\end{align*}
As also $\beta_k\ell\Vert X_k\Vert^2\leq \sqrt{2/\pi}$, inserting this inequality back into \eqref{eq: afgd mse 1} yields
\begin{align*}
  &   \E[(\omega_{k,\ell}-\theta_\star)(\omega_{k,\ell}-\theta_\star)^\top\vert X_k]
 \\
 &\leq_{\mathsf{L}}
  \E\Big[(\omega_{k,\ell-1}-\theta_\star)(\omega_{k,\ell-1}-\theta_\star)^\top\vert X_k\Big]-\sqrt{2/\pi}\beta_kX_kX_k^\top\E\Big[(\omega_{k,r-1}-\theta_\star)(\omega_{k,\ell-1}-\theta_\star)^\top\vert X_k\Big]
 \\
 &\quad-\sqrt{2/\pi}\beta_k\E\Big[(\omega_{k,\ell-1}-\theta_\star)(\omega_{k,\ell-1}-\theta_\star)^\top \vert X_k\Big]X_kX_k^\top
 \\&\quad+\beta_k^2\Vert X_k\Vert^2\Big((1-\sqrt{2/\pi}\beta_k \Vert X_k\Vert^2)^{\ell-1} X_k^\top\E\Big[(\omega_{k-1,\ell}-\theta_\star)(\omega_{k-1,\ell}-\theta_\star)^\top \Big]X_k
 \\&\quad+\frac{4}{\pi}\beta_k^2\Vert X_k\Vert^4\ell^2\Big)\mathbb{I}_d
+\beta_k^2\Big(\frac{4}{\pi}\ell X_kX_k^\top
 +\Vert X_k\Vert^2\mathbb{I}_d\Big)
  \\
 &\leq_{\mathsf{L}}
  \Big(\mathbb{I}_d-\sqrt{2/\pi}\beta_kX_kX_k^\top\Big)\E\Big[(\omega_{k,\ell-1}-\theta_\star)(\omega_{k,\ell-1}-\theta_\star)^\top\vert X_k\Big]\Big(\mathbb{I}_d-\sqrt{2/\pi}\beta_kX_kX_k^\top\Big)
 \\&\quad+\beta_k^2\Vert X_k\Vert^2X_k^\top\E\Big[(\omega_{k-1,\ell}-\theta_\star)(\omega_{k-1,\ell}-\theta_\star)^\top \Big]X_k
 \mathbb{I}_d
\\&\quad
+\beta_k^2\Big(\frac{4}{\pi}\ell X_kX_k^\top
 +\Vert X_k\Vert^2\mathbb{I}_d+\frac{4}{\pi}\beta_k^2\Vert X_k\Vert^6 \ell^2\mathbb{I}_d\Big)
 \\
 &\leq_{\mathsf{L}}
  \Big(\mathbb{I}_d-\sqrt{2/\pi}\beta_kX_kX_k^\top\Big)^\ell\E\Big[(\omega_{k-1,\ell}-\theta_\star)(\omega_{k-1,\ell}-\theta_\star)^\top\Big]\Big(\mathbb{I}_d-\sqrt{2/\pi}\beta_kX_kX_k^\top\Big)^\ell
 \\&\quad+\beta_k^2\Vert X_k\Vert^2X_k^\top\E\Big[(\omega_{k-1,\ell}-\theta_\star)(\omega_{k-1,\ell}-\theta_\star)^\top \Big]X_k
\sum_{i=0}^{\ell-1}\Big(\mathbb{I}_d-\sqrt{2/\pi}\beta_kX_kX_k^\top\Big)^{2i}
\\&\quad
+\beta_k^2\Big(\frac{4}{\pi}\ell X_kX_k^\top
 +2\Vert X_k\Vert^2\mathbb{I}_d\Big)\sum_{i=0}^{\ell-1}\Big(\mathbb{I}_d-\sqrt{2/\pi}\beta_kX_kX_k^\top\Big)^{2i}
  \\
 &\leq_{\mathsf{L}}
  \Big(\mathbb{I}_d-\sqrt{2/\pi}\beta_kX_kX_k^\top\Big)^\ell\E\Big[(\omega_{k-1,\ell}-\theta_\star)(\omega_{k-1,\ell}-\theta_\star)^\top\Big]\Big(\mathbb{I}_d-\sqrt{2/\pi}\beta_kX_kX_k^\top\Big)^\ell
 \\&\quad+\beta_k^2\ell\Vert X_k\Vert^2X_k^\top\E\Big[(\omega_{k-1,\ell}-\theta_\star)(\omega_{k-1,\ell}-\theta_\star)^\top \Big]X_k\mathbb{I}_d
\\&\quad
+\beta_k^2\ell\Big(\frac{4}{\pi}\ell X_kX_k^\top
 +2\Vert X_k\Vert^2\mathbb{I}_d\Big).
\end{align*}
Arguing similarly as in \eqref{eq: 1} gives
\begin{align*}
  \Big(\mathbb{I}_d-\sqrt{2/\pi}\beta_kX_kX_k^\top\Big)^{\ell}
  &\leq_{\mathsf{L}}\mathbb{I}_d-\frac{\ell\beta_k}{\sqrt{2\pi}}X_kX_k^\top,
\end{align*}
and because $\ell\sqrt{2/\pi}\beta_kb\leq 3/2$, the above inequalities give 
\begin{align*}
  &   \operatorname{MSE}(\omega_{k,\ell})
  \\
 &\leq
  \tr\Big(\E\Big[\Big(\mathbb{I}_d-\sqrt{2/\pi}\beta_kX_kX_k^\top\Big)^{2\ell}\Big]\E\Big[(\omega_{k-1,\ell}-\theta_\star)(\omega_{k-1,\ell}-\theta_\star)^\top\Big]\Big)
 \\&\quad+\beta_k^2d\ell b\tr\Big(\Sigma \E\Big[(\omega_{k-1,\ell}-\theta_\star)(\omega_{k-1,\ell}-\theta_\star)^\top \Big] \Big)
\\&\quad
+\beta_k^2\tr(\Sigma)\ell\Big(\frac{4}{\pi}\ell 
 +2d\Big)
  \\
 &\leq
  \lambda_{\max}\Big(\E\Big[\mathbb{I}_d-\frac{\ell\beta_k}{\sqrt{2\pi}}X_kX_k^\top\Big]\Big)\operatorname{MSE}(\omega_{k-1,\ell})
+\beta_k^2d\ell b\lambda_{\max}(\Sigma)\operatorname{MSE}(\omega_{k-1,\ell})
\\&\quad
+\beta_k^2\tr(\Sigma)\ell\Big(\frac{4}{\pi}\ell 
 +2d\Big)
   \\
 &=
\Big(1-\frac{\ell\beta_k}{\sqrt{2\pi}}\lambda_{\min}(\Sigma)+\beta_k^2d\ell b\lambda_{\max}(\Sigma)\Big)\operatorname{MSE}(\omega_{k-1,\ell})
+\beta_k^2\tr(\Sigma)\ell\Big(\frac{4}{\pi}\ell 
 +2d\Big).
\end{align*}
Noting that $\beta_kdb\lambda_{\max}(\Sigma)\leq\lambda_{\min}(\Sigma)/\sqrt{8\pi}, $ arguing as in the proof of Theorem \ref{thm: mspe} and denoting $c=c_1c_2$ yields
\begin{align*}
  &   \operatorname{MSE}(\omega_{k,\ell})
   \\
 &\leq
\operatorname{MSE}(\omega_{0})\prod_{i=1}^k\Big(1-\frac{\ell\beta_i}{\sqrt{8\pi}}\lambda_{\min}(\Sigma)\Big)
\\&\quad+\sum_{i=1}^k\beta_i^2\tr(\Sigma)\ell\Big(\frac{4}{\pi}\ell+2d\Big) \prod_{m=i+1}^k\Big(1-\frac{\ell\beta_m}{\sqrt{8\pi}}\lambda_{\min}(\Sigma)\Big)
   \\
 &\leq\left(\frac{1+c}{k+1+c}\right)^2
\operatorname{MSE}(\omega_{0})
+\tr(\Sigma)\ell\Big(\frac{4}{\pi}\ell+2d
\Big)\sum_{i=1}^k\beta_i^2 \left(\frac{i+1+c}{k+1+c}\right)^2
   \\
 &\leq\left(\frac{1+c}{k+1+c}\right)^2
\operatorname{MSE}(\omega_{0})
+8c_1^2\tr(\Sigma)\Big(\frac{2}{\pi}+\frac{d}{\ell}\Big)\frac{k}{(k+1+c)^2},
\end{align*}
completing the proof. 
\end{proof}
\printbibliography
\end{document}